\documentclass[12pt]{article}

\usepackage{hyperref}
\usepackage{amssymb,amsmath,amsfonts,graphicx,color,multicol,amsthm}
\usepackage{subcaption}
\usepackage[hmargin=1.0in]{geometry}
\usepackage[T1]{fontenc}

\newcommand{\Nn}{{\mathbb{N}}}
\newcommand{\al}{\alpha}
\newcommand{\be}{\beta}
\newcommand{\W}{\mathrm{W}}
\newcommand{\Wn}{\mathrm{W}_0}
\newcommand{\Wm}{\mathrm{W}_{-1}}

\newtheorem{thm}{Theorem}[section]
\newtheorem{lem}[thm]{Lemma}
\newtheorem{cor}[thm]{Corollary}

\newtheorem{rem}[thm]{Remark}
\newtheorem{conj}[thm]{Conjecture}
\newtheorem{quest}[thm]{Question}

\usepackage{amssymb}

\begin{document}

\title{Explicit and recursive estimates of the Lambert W function}

\author{Lajos L\'oczi\thanks{{\texttt{LLoczi@inf.elte.hu}}, Department of Numerical Analysis, E\"otv\"os Lor\'and University, and Department of Differential Equations, Budapest University of Technology and Economics, Hungary. 
The project ``Application-domain specific highly reliable IT solutions'' has been implemented with the support provided from the National Research, Development and Innovation Fund of Hungary, financed under the Thematic Excellence Programme TKP2020-NKA-06 (National Challenges Subprogramme) funding scheme.}}
\date{\today}

\maketitle

\begin{abstract}
Solutions to a wide variety of transcendental equations can be expressed in terms of
the Lambert $\W$ function. The $\W$ function, occurring frequently in applications, is a non-elementary,  but now standard mathematical function implemented in all major technical computing systems. In this work, we discuss some approximations of the two real branches, $\Wn$ and $\Wm$. On the one hand, we present some analytic lower and upper bounds on $\Wn$ for large arguments that improve on some earlier results in the literature. On the other hand, we analyze two logarithmic recursions, one with linear, and the other with quadratic rate of convergence. We propose suitable starting values for the recursion with quadratic rate that ensure convergence on the whole domain of definition of both real branches. We also provide  \emph{a priori}, simple, explicit and uniform estimates on its convergence speed that enable guaranteed, high-precision approximations of $\Wn$ and $\Wm$ at any point. Finally, as an application of the $\Wn$ function, we settle a conjecture about the growth rate of the positive non-trivial solutions to the equation $x^y=y^x$.  
\end{abstract}

\noindent \textbf{Keywords:} Lambert W function; explicit estimates; recursive approximations

\section{Introduction}\label{intro}

The Lambert $\W$ function---first investigated in the 18$^{\text{th}}$ century---is defined implicitly by the transcendental equation
\[
\W(x)e^{\W(x)}=x.
\]
It has now become a standard mathematical function and it is included in all major technical computing systems. 
It appears in an increasingly growing number of applications (see, e.g., \cite{boyd} and the references therein) due to the fact that solutions to a wide variety of polynomial-exponential-logarithmic equations can be expressed in terms of the $\W$ function. 

The $\W$ function has two real, and infinitely many complex branches \cite{johansson}. The real branches are usually denoted by
\[
\Wn:[-1/e,\infty)\to[-1,\infty)
\]
and
\[
\Wm:[-1/e,0)\to(-\infty,-1], 
\]
see Figure \ref{figbranch}. Both of these are strictly monotone, and some simple special values include
$\Wn(0)=0$, $\Wn(e)=1$, $\Wn(-1/e)=-1$, or $\Wm(-1/e)=-1$.
\begin{figure}
\begin{center}
\includegraphics[width=0.66\textwidth]{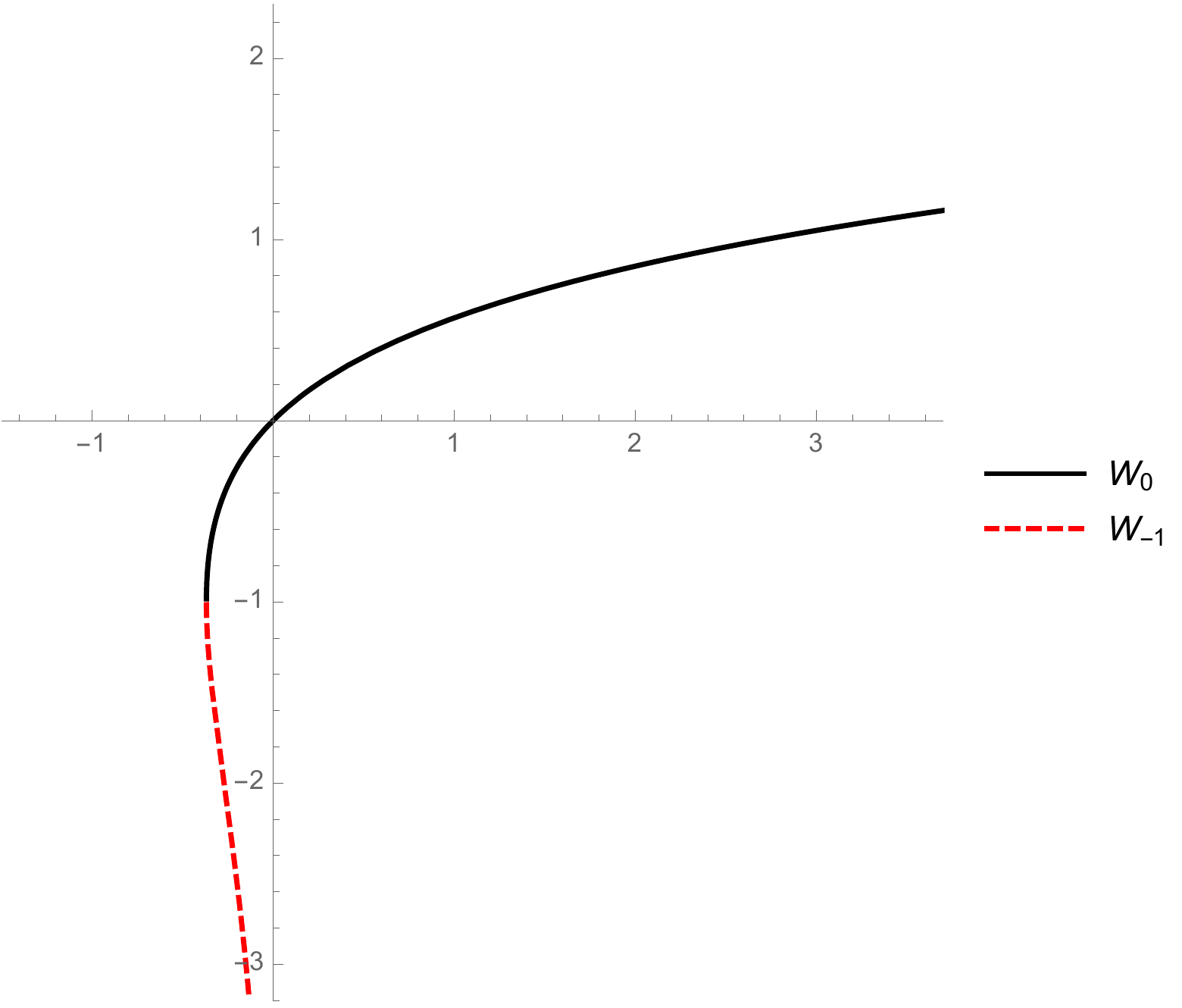}
\caption{The real branches of the $\W$ function
\label{figbranch}}
\end{center}
\end{figure}

The $\W$ function is not an elementary function \cite{bronstein}, so it is natural to ask how one can approximate it efficiently with simpler functions. 
In the literature, one can find many different representations and approximations for the real branches of the $\W$ function on various intervals, see, e.g., \cite{boyd,johansson,wolframsite,hoofar, faris} and the references therein. These include
\begin{itemize}
\item[(i)] series expansions
\begin{itemize}
\item[$\bullet$] Taylor expansions, e.g., about the origin
\begin{equation}\label{W00series}
\sum_{k=1}^{\infty} \frac{(-k)^{k-1}}{k!}x^k=x-x^2+\frac{3 x^3}{2}-\frac{8 x^4}{3}+\frac{125 x^5}{24}+{\mathcal{O}}\left(x^6\right);
\end{equation}
\item[$\bullet$]  Puiseux expansions, e.g., about the branch point $x=-1/e$;\\ 
\item[$\bullet$]  asymptotic expansions about $+\infty$, such as
\begin{equation}\label{asymptexp}
\ln(x)-\ln(\ln(x))+\sum_{k=0}^{\infty}\sum_{m=1}^{\infty} c_{k,m}\frac{(\ln(\ln(x)))^m}{(\ln(x))^{m+k}}
\end{equation}
where the coefficients $c_{k,m}$ are defined in terms of the Stirling cycle numbers;
\end{itemize}
\item[(ii)] recursive approximations
\begin{itemize}
\item[$\bullet$] the recursion 
\begin{equation}\label{lambdarecursion}
\lambda_{n+1}(x):=    \ln(x)-\ln(\lambda_n(x));
\end{equation}
\item[$\bullet$] the Newton-type iteration
\begin{equation}\label{Newtonrecursion}
\nu_{n+1}(x):=\nu_n(x)-\frac{\nu_n(x)-x e^{-\nu_n(x)}}{1+\nu_n(x)};
\end{equation}
\item[$\bullet$] the iteration
\begin{equation}\label{IBrecursion}
\be_{n+1}(x):=\frac{\be_n(x)}{1+\be_n(x)}\left(1+\ln\left(\frac{x}{\be_n(x)}\right)\right);
\end{equation}
\item[$\bullet$] the Halley-type iteration
\begin{equation}\label{Halleyrecursion}
h_{n+1}(x):=h_n(x)-\frac{h_n(x) e^{h_n(x)} -x}{e^{h_n(x)} \left(h_n(x)+1\right)-\frac{\left(h_n(x)+2\right) 
\left(h_n(x) e^{h_n(x)} -x\right)}{2 \left(h_n(x)+1\right)}};
\end{equation}
\item[$\bullet$] the Fritsch--Shafer--Crowley (FSC) scheme 
\begin{equation}\label{FSCrecursion}
f_{n+1}(x):=f_n(x)\left(1+\frac{z_n(x)[q_n(x)-z_n(x)]}{(1+f_n(x))[q_n(x)-2z_n(x)]}\right)
\end{equation}
with
\[
z_n(x):=\ln\left(\frac{x}{f_n(x)}\right)-f_n(x) \quad\text{and}\quad q_n(x):=2(1+f_n(x))\left(1+f_n(x)+\frac{2}{3}z_n(x)\right);
\]
\end{itemize}
\item[(iii)] analytic bounds on different intervals
\begin{itemize}
\item[$\bullet$] the bounds 
\begin{equation}\label{olderbound}
\ln(x)-\ln(\ln(x))+\frac{\ln(\ln(x))}{2\ln(x)}<\Wn(x)<\ln(x)-\ln(\ln(x))+\frac{e \ln(\ln(x))}{(e-1)\ln(x)},
\end{equation}
valid for $x\in(e,+\infty)$;
\item[$\bullet$] or, for example, the bounds 
\begin{equation}\label{W-1bound}
\frac{e \ln(-x)}{e-1}\le\Wm(x)\le \ln(-x)-\ln(-\ln(-x))
\end{equation}
valid for $x\in [-1/e,0)$.
\end{itemize}
\end{itemize}

As for the other (complex) branches of the $\W$ function, \cite{johansson} contains an algorithm to approximate any branch by using complex interval arithmetic together with the Arb library.\\

Now let us comment on some of the above formulae to motivate our work. 

The recursion \eqref{lambdarecursion} is based on the functional equation \eqref{functional}, and 
has appeared many times in the literature. 
The recursion \eqref{IBrecursion} was devised in \cite{boyd} (we only changed their notation from $\W_n$ to $\be_n$, since $\W_n$ usually denotes the complex branches of the $\W$ function); moreover, the authors mention that its convergence rate is quadratic, and it approximates  
$\Wn(x)$ for large $x$ better than the standard (also quadratic) Newton iteration \eqref{Newtonrecursion}. 
The Halley recursion \eqref{Halleyrecursion} has third order of convergence (in general, \eqref{Newtonrecursion} and  \eqref{Halleyrecursion} both belong to the Schr\"oder families of root-finding methods, see, e.g., \cite{petkovic}), and the FSC scheme  \eqref{FSCrecursion} converges at an even faster rate. However, the rate of convergence of \eqref{lambdarecursion} has not yet been investigated, and, more importantly,  suitable starting values have not been reported in the literature guaranteeing that these recursions are well-defined, 
nor explicit bounds on the error committed in the $n^\text{th}$ step.

The pair of bounds \eqref{olderbound}---based on the initial terms of the series \eqref{asymptexp}---appears in \cite{hoofar} (its weaker version is reproduced in our Lemma \ref{lemmaWsimple} below). For $x>e$, \cite[Section 4.3]{boyd} describes some tighter, two-sided bounds for $\Wn(x)$, obtained by applying one step of \eqref{IBrecursion} or \eqref{lambdarecursion} to a suitable initial function. These bounds contain more nested logarithms (hence, they are not of the form \eqref{asymptexp}).

Finally, when dealing with various expansions (Taylor, Puiseux or asymptotic series in the group (i) above)  in practice, one can work only with their finite truncations, so one also needs estimates of the remainder terms---estimates of this type were published only very recently \cite{johansson}.




\subsection{Summary of the results and structure of the paper}

In Section \ref{refinedapproximationsW}, we present some two-sided, explicit estimates of $\Wn(x)$ for large values of $x$.
The structure of these estimates is based on the first few terms of \eqref{asymptexp} (but their proofs do not rely on the asymptotic series). 
These results strictly refine the estimate \eqref{olderbound} for any $x>e$; moreover, instead of having an error term ${\mathcal{O}}\left(\frac{\ln(\ln (x))}{\ln (x)}\right)$ as in \eqref{olderbound}, our error terms have the form 
${\mathcal{O}}\left(\left(\frac{\ln(\ln (x))}{\ln(x)}\right)^3\right)$ and ${\mathcal{O}}\left(\frac{\ln^2(\ln (x))}{\ln^3 (x)}\right)$. 

In Section \ref{lambdarecursionsection}, we analyze the recursion \eqref{lambdarecursion} for $x>e$ large enough.  By providing a simple starting value, we show that its even- and odd-indexed subsequences converge to $\Wn(x)$ from above and below, respectively. More importantly, we give an explicit error estimate for the linear rate of convergence.

In Section \ref{IBsection}, a complete analysis of the recursion \eqref{IBrecursion} is given. Here, we propose simple and suitable starting values 
(consisting of the basic operations, logarithms, or  square roots) that guarantee monotone convergence on the full domain of definition of both real branches: 
for the branch $\Wn$ on $(e,+\infty)$, $(0,e)$, and $(-1/e,0)$, as well as for the branch $\Wm$ on 
$(-1/e,0)$. Again, the essential feature of these theorems is that the quadratic rate of convergence of \eqref{IBrecursion} is proved via explicit and uniform error estimates.
Thanks to their simplicity, the maximum number of iteration steps needed to achieve a desired precision can easily be determined in advance. We also reproduce some guaranteed, high-precision approximations of $\Wn$ in \emph{Mathematica} that were computed in a different software environment and reported in \cite{johansson}---for very large arguments (so large that their direct evaluation in  \emph{Mathematica} via its built-in function {\texttt{ProductLog}} is not possible), or for arguments very close to the branch point $x=-1/e$.

Finally, in Section \ref{sectionxyyx}, we present a simple application of the $\Wn$ function and settle a conjecture in \cite{gofen} about the growth rate of the non-trivial positive solutions of $x^y=y^x$.

To make the presentation of our results easier, all technical proofs of the theorems and lemmas are collected in Appendix \ref{proofsection}. The proofs are almost entirely of symbolic character.  As for the techniques, monotonicity arguments are typical. To tackle transcendental inequalities (e.g., ones with roots, exponential functions and logarithms simultaneously), repeated differentiation and various substitutions are used to convert them to inequalities containing only rational functions or (multivariable) polynomials, whose behavior is easier to analyze.

\subsection{Notation and some preliminary results}

The set of natural numbers is denoted by $\mathbb{N}:=\{0,1,2,\ldots\}$, and the abbreviations  
\[
L_1:=\ln \quad\text{and}\quad L_2:=\ln\circ\ln
\]
will often appear. Auxiliary objects in the proofs will sometimes carry subscripts referring to the number of the (sub)section in which they appear (for example, the polynomial $P_{\ref{lemma9007lemmaproofsection}}$ and the set ${\mathcal{S}}_{\ref{lemma9007lemmaproofsection}}$ both appear in Section \ref{lemma9007lemmaproofsection}).
 
Next, we collect some elementary results which will also be used later.
\begin{itemize}
\item[(i)] From the definition of $\Wn$ and $\Wm$, it is easily seen that the following identities are 
satisfied:
\begin{equation}\label{functional}
\Wn(x)=\ln(x)-\ln(\Wn(x))\quad \text{for } x\in(0,+\infty);
\end{equation}
\begin{equation}\label{functional2}
\Wn(x)=-\ln\left(\frac{\Wn(x)}{x}\right)\quad \text{for } x\in(-1/e,0);
\end{equation}
\begin{equation}\label{functional3}
\Wm(x)=-\ln\left(\frac{\Wm(x)}{x}\right)\quad \text{for } x\in(-1/e,0).
\end{equation}
\item[(ii)] The strict monotonicity of the function $[-1,+\infty)\ni x\mapsto x e^x$ implies that 
for any $\alpha, \beta\in [-1,+\infty)$ we have
\begin{equation}\label{monotoneequivalence}
\alpha \ \boxed{ \lesseqqgtr} \ \beta\text{\ if and only if } \alpha e^\alpha \ \boxed{ \lesseqqgtr} \ \beta e^\beta,
\end{equation}
where $\boxed{ \lesseqqgtr}$ is either ``$<$'', or ``$=$'', or ``$>$''.
\item[(iii)] The following auxiliary inequality appears in \cite{hoofar}; for the sake of completeness, we reprove it in Section \ref{sectionWsimple}. 
\begin{lem}\label{lemmaWsimple} On $(e,+\infty)$ we have
$L_1-L_2<\Wn<L_1$,
and $L_1(e)-L_2(e)=\Wn(e)=L_1(e)=1$.
\end{lem}
\end{itemize}

\section{Refined lower and upper bounds for $\Wn$ for large arguments}\label{refinedapproximationsW}

The main result of the present section is Theorem \ref{maintheorem} below. After setting up the particular form of the lower and upper estimates in this theorem, the domains of the corresponding inequalities have been optimized. To describe these
domains, first we define two constants, $x^*$ and $x^{**}$,  with the help of the following lemmas. 
\begin{lem}\label{9007lemma}
For $y\ge 1$ we set
$
f_{1}(y):=y^{1/y} \left(y^2-y \ln (y)+\ln(y)\right)-y^2$.
Then there is a unique $y^*>1$ such that $f_1(y^*)=0$. We have 
\[
f_1 \  \left\{
\begin{aligned}
  > 0  & \quad\quad \text{ on } (1,y^*), \\ 
  < 0  &  \quad\quad \text{ on } (y^*,+\infty), \\
  = 0 &  \quad\quad \text{ at the points } \{1, y^*\}. \\
\end{aligned}
\right.
\]
\end{lem}
Now we define $x^*:=e^{y^*}\approx 6288.69$.
\begin{lem}\label{13.26lemma}
For $y\ge 1$ we set
$
f_{2}(y):=(y-3) \ln(y)-2 y$.
Then there is a unique $y^{**}>1$ such that $f_2(y^{**})=0$. We have 
\[
f_2 \  \left\{
\begin{aligned}
  < 0  & \quad\quad \text{ on } [1,y^{**}), \\ 
  > 0  &  \quad\quad \text{ on } (y^{**},+\infty). \\
  \end{aligned}
\right.
\]
\end{lem}
We define 
$
x^{**}:=e^{y^{**}}\approx 573967.06$.\\

Lemma \ref{9007lemma} is proved in
Section \ref{lemma9007lemmaproofsection}, and the elementary proof of Lemma \ref{13.26lemma} is given in Section \ref{section34}. We can now formulate the following result.
\begin{thm}\label{maintheorem}
We have
$
L_1-L_2+\frac{L_2}{L_1}\ \boxed{ \lesseqqgtr}\ \Wn
$,
with $\boxed{ \lesseqqgtr}$ defined as 
\[
\left\{
\begin{aligned}
  > &  \quad\quad \text{ on } (e, x^*), \\
  = &  \quad\quad \text{ at the points } \{e, x^*\}, \\
  <  & \quad\quad \text{ on } (x^*,+\infty). \\ 
\end{aligned}
\right.
\]
On $(x^*,+\infty)$, we have the upper estimate
\begin{equation}\label{maintheoremupperestimate}
\Wn<L_1-L_2+\frac{L_2}{L_1}+\frac{(L_2-2) L_2}{2 L_1^2}+
\frac{L_2^3}{L_1^3}.
\end{equation}
On the interval $(x^{**},+\infty)$, the lower bound on $\Wn$ can be improved to
\begin{equation}\label{lowboundimproved}
L_1-L_2+\frac{L_2}{L_1} \ \boxed{<}\  L_1-L_2+\frac{L_2}{L_1}+\frac{(L_2-2) L_2}{2 L_1^2}-\frac{3L_2^2}{2L_1^3}<\Wn,
\end{equation}
and the smallest number $x_0>e$  such that $\boxed{<}$ holds on 
$(x_0,+\infty)$ is $x_0=x^{**}$.
\end{thm}
\begin{rem}
It can be shown that the estimates in Theorem \ref{maintheorem} are strictly sharper than those in \eqref{olderbound} for any $x>e$.
\end{rem}

The proof of Theorem \ref{maintheorem} is given in Section \ref{section34}, and it relies on the identity \eqref{functional}, and on the fact that the function $-L_1$ is strictly decreasing: if one has a lower estimate of $\Wn$, then \eqref{functional} yields an upper estimate, and vice versa.

By repeatedly applying this bootstrap procedure, we obtain the sequence of two-sided estimates presented in Section \ref{lambdarecursionsection}. In Theorem \ref{maintheorem}, the bootstrap argument is used only two times. In any case, logarithms nested to several levels will soon appear. The estimates
\eqref{maintheoremupperestimate}--\eqref{lowboundimproved} have been devised to contain only $L_1$ and $L_2$, and to conjecture them, the first few terms of the asymptotic expansion \eqref{asymptexp} 
have been used.

For a recent, related and general result, see \cite[Theorem 2]{johansson}. In that theorem, an error term in explicit form is given when the double series in the asymptotic expansion \eqref{asymptexp} is truncated at some indices, and the modulus of the argument of the $\W$ function is sufficiently large. Our Theorem \ref{maintheorem} presents some simple explicit lower and upper bounds for the $\Wn$ branch. The proof of Theorem \ref{maintheorem} is a direct one, and is independent of the proof of \cite[Theorem 2]{johansson}---that proof relies on the convergence of the asymptotic series \eqref{asymptexp} on a certain subset of the complex plane.

\section{A linearly convergent recursion for $\Wn$ for large arguments}\label{lambdarecursionsection}

In this section, we analyze the recursion \eqref{lambdarecursion}: with some starting value to be proposed below, an explicit, linear convergence estimate is proved for large enough arguments. 

For any $x\ge e$ and $n\in\mathbb{N}$
let us define
\begin{equation}\label{lambdarecursiondef}
\left\{
\begin{aligned}
  \lambda_0(x):=  & \ln(x), \\ 
  \lambda_{n+1}(x):=  &  \ln(x)-\ln(\lambda_n(x)). 
 \end{aligned}
\right.
\end{equation}
\noindent Clearly,  $\lambda_n(e)=1=\Wn(e)$ for all $n\in\mathbb{N}$. For $x>e$, the lemma below shows
that $\lambda_n$ is well-defined, and its even and odd subsequences ``sandwich'' the
Lambert function.
\begin{lem}\label{sandwichlemma}
For any fixed $x>e$ and $n\in\mathbb{N}$, 
the number $\lambda_n(x)$ is real, and satisfies
\begin{equation}\label{lambdatrivialbounds}
1<\lambda_n(x)<\frac{x}{e}
\end{equation}
and
\begin{equation}\label{lambdasandwich}
\lambda_{2n+1}(x)< \Wn(x)< \lambda_{2n}(x).
\end{equation}
\end{lem}
The proof of the lemma is found in Section \ref{sandwichlemmaproof}. The main result of the present section is the following theorem about the convergence and convergence speed of the
recursion \eqref{lambdarecursiondef}. The constant $x^{***}\in (5.580,5.581)$
appearing in the theorem is 
the unique solution $x>e$ to the equation \[L_1(x)-L_2(x)=\sqrt{2\ln(2)};\] hence, for $x>x^{***}$ we have 
$\frac{\sqrt{2\ln(2)}}{\ln(x)-\ln(\ln(x))}\in(0,1)$.

\begin{thm}\label{lambdatheorem} 
Let us fix any $x>x^{***}$. Then the sequence $\lambda_{n}$ defined by \eqref{lambdarecursiondef} converges and $\displaystyle\lim_{n\to+\infty}\lambda_n(x)=\Wn(x)$. Moreover,  
for any $n\in\mathbb{N}$ we have the error estimate
\begin{equation}\label{lambdaestimate}
0< \lambda_{2n}(x)-\Wn(x)\le \left(\frac{\sqrt{2\ln(2)}}{\ln(x)-\ln(\ln(x))}\right)^{2n}\ln(\ln(x)).
\end{equation}
\end{thm}
The proof of this theorem is given in Section \ref{sectionlambdatheoremproof}. Now, by combining \eqref{lambdasandwich} and \eqref{lambdaestimate}, the following result is obtained.
\begin{cor}
For any given $x>5.581$ and tolerance $\varepsilon>0$, let us choose $n$ such that
\[
0<\left(\frac{\sqrt{2\ln(2)}}{L_1(x)-L_2(x)}\right)^{2n}L_2(x)<\varepsilon.
\] 
Then
\[
\Wn(x)\in \Big[\lambda_{2n}(x)-\varepsilon , \lambda_{2n}(x)\Big).
\] 
\end{cor}

It is also seen that the sequence $\lambda_{n}(x)$ approximates $\Wn(x)$ efficiently for large arguments: for each fixed $x>5.581$, the right-hand side 
of \eqref{lambdaestimate} converges to $0$
exponentially fast as $n\to +\infty$, and the speed of convergence improves as $x$ is chosen 
closer and closer to $+\infty$. 

\begin{rem}
Regarding the estimate \eqref{lambdaestimate}, we actually prove a slightly stronger statement in Section \ref{sectionlambdatheoremproof}, and the constraint $x>5.581$ could also be relaxed, see Lemma \ref{xtildestarlemma}. However, the estimate given in \eqref{lambdaestimate} is more explicit since its 
right-hand side does not contain $\Wn$. 
On the other hand, with some more work, one can prove that $\lambda_n(x)$ converges to $\Wn(x)$ also for $x\in(e,x^{***}]$, but this will not be pursued in the present paper because Section \ref{IBsection} will describe a more effective recursion.
\end{rem}

\begin{rem}
Numerical experiments indicate that for any fixed $x>e$ we have
\begin{equation}\label{Wlimitrelation}
\lim_{n\to+\infty} \frac{\lambda_n(x)-\Wn(x)}{\Wn(x)-\lambda_{n+1}(x)}=\Wn(x).
\end{equation}
In fact, Theorem \ref{lambdatheorem} was motivated by discovering \eqref{Wlimitrelation} first. Now as we know
that $\lambda_n$ converges pointwise to $\Wn$ on, say, $\left(x^{***}, +\infty\right)$, we can easily prove \eqref{Wlimitrelation} on this interval. Indeed, let us fix any $x>x^{***}$ and notice that
 $\Wn(x)-\lambda_{n+1}(x)\ne 0$ due to Lemma \ref{sandwichlemma}. Then the definition $\lambda_{n+1}(x)=
\ln(x)-\ln(\lambda_n(x))$ implies $\lambda_n(x)=x \exp(-\lambda_{n+1}(x))$, and from the definition of $\Wn$ we have $\Wn(x)=x\exp(-\Wn(x))$. Therefore, by using $-\lambda_{n+1}\to -\Wn$ and the differentiability of $\exp$, we get
\[
\frac{\lambda_n(x)-\Wn(x)}{\Wn(x)-\lambda_{n+1}(x)}=
x\cdot\frac{\exp(-\lambda_{n+1}(x))-\exp(-\Wn(x))}{-\lambda_{n+1}(x)-(-\Wn(x))}\to x \exp'(-\Wn(x))=\Wn(x)
\]
as $n\to +\infty$, completing the proof of \eqref{Wlimitrelation} for $x>x^{***}$.
\end{rem}

\section{A quadratically convergent recursion for $\Wn$ and $\Wm$ on their full domains of definition}\label{IBsection}

In this section, we analyze the recursion \eqref{IBrecursion} by proposing some starting values on each subinterval, then prove explicit, quadratic convergence estimates.

\subsection{Convergence to $\Wn$ on the interval $(e,+\infty)$}\label{IBconvx>e}

Due to $\Wn(e)=1$, let us fix an arbitrary $x>e$ in this section. Here we propose the following starting value:
\begin{equation}\label{x>estart}
\left\{
\begin{aligned}
  \be_0(x):=  & \ln(x)-\ln(\ln(x)), \\ 
  \be_{n+1}(x):=  &  \frac{\be_n(x)}{1+\be_n(x)}\left(1+\ln\left(\frac{x}{\be_n(x)}\right)\right)\quad (n\in\mathbb{N}). 
\end{aligned}
\right.
\end{equation}

\begin{lem}\label{lem41}
For any $x>e$, the recursion \eqref{x>estart}  satisfies 
\[
0<\be_n(x)<\be_{n+1}(x)<\Wn(x)\quad (n\in\Nn).
\]
\end{lem}

The proof of this lemma is found in Section \ref{section411stproof}. The lemma says, in particular, that the recursion \eqref{x>estart} is well-defined and real-valued. In the remainder of Section \ref{IBconvx>e}, we show that
\begin{equation}\label{section41conv}
\lim_{n\to+\infty}\be_n(x)=\Wn(x).
\end{equation}
We prove the convergence by giving some explicit error estimates as follows. 

We start with the inductive step. The proof of the following lemma is given in Section \ref{section412ndproof}. 
\begin{lem}\label{sect41lem42}
For any $x>e$ and $n\in\mathbb{N}$, we have
\begin{equation}\label{sect41lem42ineq}
0<\Wn(x)-\be_{n+1}(x)<\frac{(\Wn(x)-\be_n(x))^2}{(1+\be_n(x))\Wn(x)}.
\end{equation}
\end{lem}

The next lemma describes some simple estimates for the starting value. Its proof is found in Section \ref{section413rdproof}.
\begin{lem}\label{sect41lem43}
For any $x>e$, one has
\begin{equation}\label{lem431stest}
0<\Wn(x)-\be_0(x)<\frac{e}{e-1}\frac{\ln(\ln (x))}{\ln(x)}.
\end{equation}
In particular, with $\kappa_1:=\ln\left(1+1/e\right)\in(0.31,0.32)$ and for any $x>e$
\begin{equation}\label{lem432ndest}
0<\Wn(x)-\be_0(x)\le \kappa_1.
\end{equation}
\end{lem}
Now we can state the main result of this section.
\begin{thm}\label{kappacorollary}
For $n\in\mathbb{N}^+$ and for any $x>e$, the recursion \eqref{x>estart} satisfies
\begin{equation}\label{cor44xest}
0<\Wn(x)-\be_n(x)<\frac{\left(\frac{e}{e-1}\frac{\ln(\ln (x))}{\ln(x)}\right)^{2^n}}{\left(\ln(x)-\ln(\ln(x))\right)^{-1+2^n}},
\end{equation}
and also the uniform estimate
\begin{equation}\label{cor44unifest}
0<\Wn(x)-\be_n(x)<\kappa_1^{2^n}<\left(\frac{32}{100}\right)^{2^n}.
\end{equation}
\end{thm}
\begin{proof} To prove \eqref{cor44xest}, one drops the factor $1+\be_n(x)>1$ from the denominator of the upper estimate in \eqref{sect41lem42ineq}, then applies it recursively to get
\begin{equation}\label{cor44auxest}
0<\Wn(x)-\be_n(x)<\frac{\left( \Wn(x)-\be_0(x)\right)^{2^n}}{\left( \Wn(x) \right)^{-1+2^n}}.
\end{equation}
Then we use \eqref{lem431stest} in the numerator and Lemma \ref{lemmaWsimple} in the denominator.
To prove \eqref{cor44unifest}, due to $\Wn(x)>1$, we drop the denominator of the upper estimate in \eqref{cor44auxest} and use \eqref{lem432ndest}.
\end{proof}
The above theorem of course also proves \eqref{section41conv}. Regarding the estimate \eqref{cor44xest}, due to Lemma \ref{elementarylemma}, we have $\frac{e}{e-1}\frac{\ln(\ln (x))}{\ln(x)}\in (0,1)$ and $\ln(x)-\ln(\ln(x))>1$ for $x>e$. Moreover, similarly to the recursion in Section \ref{lambdarecursionsection}, \eqref{cor44xest} shows that the convergence of \eqref{x>estart} becomes faster for larger and larger values of $x$. The quality of approximations appearing in Theorem \ref{kappacorollary} can be observed in Figure \ref{fig1region}.

\begin{figure}
\begin{subfigure}{.5\textwidth}
  \centering
  \includegraphics[width=.9\linewidth]{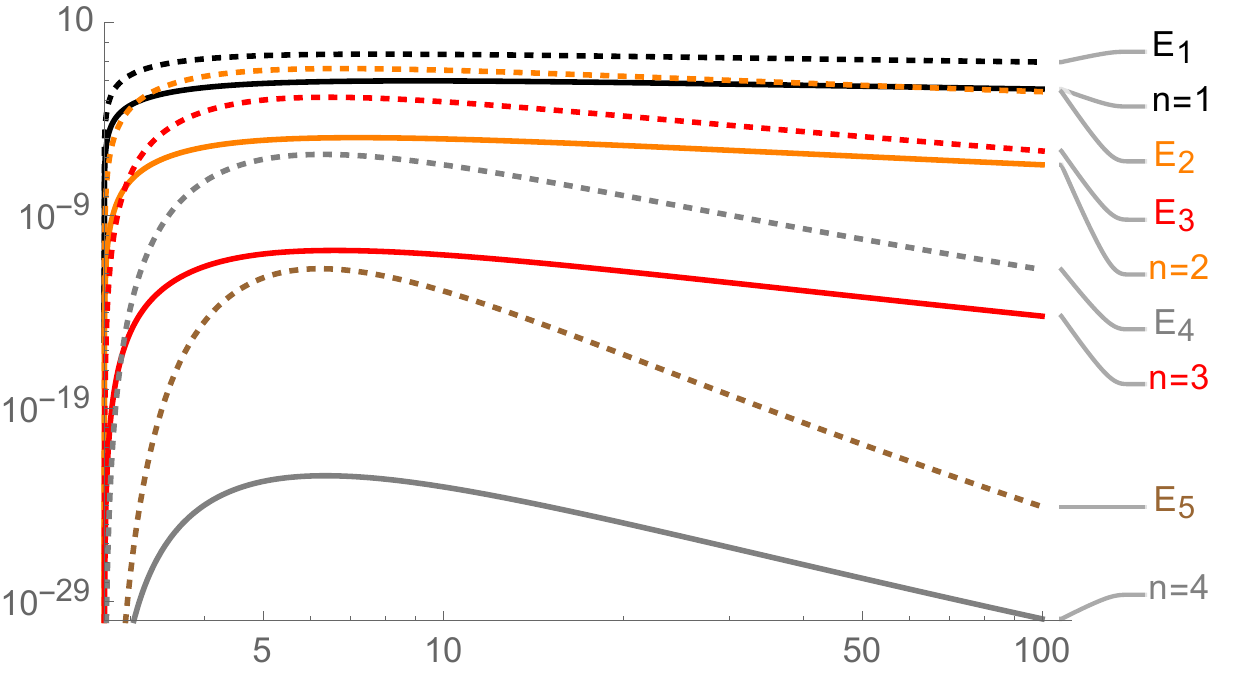}
  \caption{}
\end{subfigure}%
\begin{subfigure}{.5\textwidth}
  \centering
  \includegraphics[width=.9\linewidth]{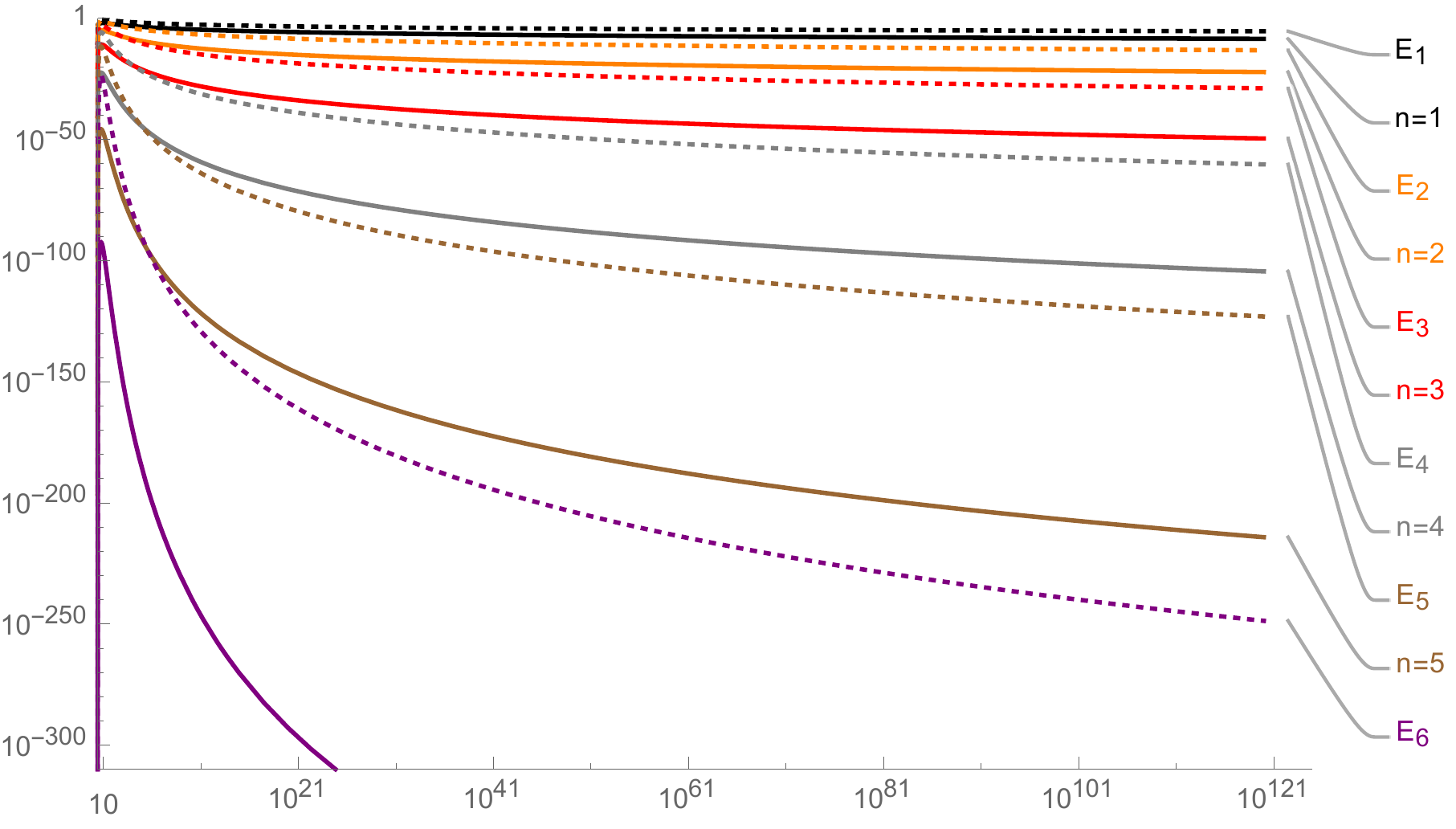}
  \caption{}
\end{subfigure}
\caption{The log-log plot in figure (a) illustrates the quantities in \eqref{cor44xest} (see Theorem \ref{kappacorollary}). The continuous curves correspond to the actual differences $\Wn(x)-\be_n(x)$ for various values of $n$, whereas the dotted curves $E_n$ depict the right-hand side of the estimate \eqref{cor44xest}.
Figure (b) shows the same expressions but for an extended range of $x$ values.}
\label{fig1region}
\end{figure}

\begin{rem} 
According to \eqref{cor44unifest}, we have the following uniform estimates for any $x>e$:
\[
0<\Wn(x)-\be_5(x)<8\cdot 10^{-17},
\]
\[
0<\Wn(x)-\be_{10}(x)<7\cdot 10^{-517},\]\[
0<\Wn(x)-\be_{15}(x)<8\cdot 10^{-16519}.
\]
\end{rem}

\begin{rem}
In \textit{Mathematica} (version 11), a direct evaluation of $\Wn\left(10^{10^3}\right)$ with its command \emph{\texttt{ProductLog}} is not possible: although the number $10^{1000}$ itself can easily be represented in this computer system, its internal algorithms cannot handle $\Wn\left(10^{1000}\right)$. (Based on the error messages, the reason is probably the following: \textit{Mathematica} uses (a variant) of the recursion \eqref{Newtonrecursion}, which 
contains the expression $x e^{-\nu_n(x)}$, and here $x>0$ is large, while $e^{-\nu_n(x)}$ is too close to $0$. Indeed, it seems that this particular piece of code tries to represent $e^{-\nu_n(x)}$ as a 
``machine number'', even if high-precision computation is requested.)

Now with the recursion \eqref{x>estart}, it is straightforward to estimate even $\Wn\left(10^{10^{20}}\right)$ in \textit{Mathematica} by taking advantage of the logarithms appearing in the starting value $\beta_0$ and rewriting $\ln\left(10^{10^{20}}\right)$ as $10^{20}\cdot \ln(10)$. In particular, due to Theorem \ref{kappacorollary} we have
\[
0<\Wn\left(10^{10^{20}}\right)-\be_{9}\left(10^{10^{20}}\right)<10^{-10000}.
\]
In fact, the difference above is even smaller than $2\cdot 10^{-19873}$, and the computation of  
$\be_{9}\left(10^{10^{20}}\right)$ to the desired precision took less than 0.33 seconds in \textit{Mathematica} on a standard laptop. These huge $\Wn$ values may have significance in number theory, because there are some estimates of the non-trivial roots of the Riemann $\zeta$ function expressed in terms of the $\Wn$ function \cite[Section 8]{boyd}.

The approximation of the quantity $\Wn\left(10^{10^{20}}\right)$ to 10000 digits of precision appears in  \cite[Section 6]{johansson}; it is implemented in the Arb library. We found that all the displayed digits of this number are in perfect agreement with the corresponding digits of our quantity $\be_{9}\left(10^{10^{20}}\right)$ computed in \textit{Mathematica}.
\end{rem}

\subsection{Convergence to $\Wn$ on the interval $(0,e)$}\label{subsection42}

Let us fix an arbitrary $0<x<e$ in this section. On this interval, we propose the following simple starting value:
\begin{equation}\label{0<x<estart}
\left\{
\begin{aligned}
  \be_0(x):=  & x/e, \\ 
  \be_{n+1}(x):=  &  \frac{\be_n(x)}{1+\be_n(x)}\left(1+\ln\left(\frac{x}{\be_n(x)}\right)\right)\quad (n\in\mathbb{N}). 
\end{aligned}
\right.
\end{equation}
\smallskip
\noindent By using the formula for the derivative of the inverse function, we have
\[
\Wn''(x)=-\frac{(\Wn(x))^2 (W(x)+2)}{x^2 (W(x)+1)^3}<0,
\]
so $\Wn$ is strictly concave on $(0,e)$, and $\Wn(x)=x/e$ holds at $x=0$ and $x=e$, hence $0<\be_0(x)<\Wn(x)$ on this interval. But this means that Lemmas \ref{lem41}--\ref{sect41lem42} and their proofs remain valid also for $x\in(0,e)$. Therefore, we can repeat the first few steps of the proof of Theorem \ref{kappacorollary} to arrive at the inequality
\begin{equation}\label{sect422nineq}
0<\Wn(x)-\be_n(x)<\frac{\left( \Wn(x)-\be_0(x)\right)^{2^n}}{\left( \Wn(x) \right)^{-1+2^n}}
\end{equation}
again ($n\in\mathbb{N}^+$). However, unlike on the interval $(e,+\infty)$ in the previous section, now the denominator of \eqref{sect422nineq} can get arbitrarily close to $0$ on $(0,e)$, so some care must be taken. First, we state the following lemma, whose proof is given in Section \ref{sect42lem47proof}.
\begin{lem}\label{sect42lem47}
For any $x\in(0,e)$ we have
\[
0<\Wn(x)-\be_0(x)<\frac{1}{5}.
\]
\end{lem}
\begin{rem} 
 There is no simple formula for the global maximum of the function $\Wn-\be_0$ on $(0,e)$ (with $\be_0$ defined in \eqref{0<x<estart}). Nevertheless, the value $1/5$ given above is close to the actual global maximum (which is approximately $0.1993$)---cf.~Lemma \ref{sect41lem43}, with $\be_0$ defined in \eqref{x>estart}, where the global maximum on $(e,+\infty)$ is exactly $\kappa_1$.
\end{rem}
The following uniform upper estimate is the main result of this section, also proving $\lim_{n\to+\infty}\be_n(x)=\Wn(x)$ for $0<x<e$.

\begin{thm}\label{thm47sect42}
With $\kappa_2:=1-1/e$ and for any $n\in\mathbb{N}^+$ and $0<x<e$, the recursion \eqref{0<x<estart} satisfies 
\begin{equation}\label{thm47unifest}
0<\Wn(x)-\be_n(x)<\frac{1}{5}\cdot\kappa_2^{-1+2^n}<\frac{1}{5}\cdot\left(\frac{633}{1000}\right)^{-1+2^n}.
\end{equation}
\end{thm}
\begin{proof} We give a simple upper estimate of the rightmost fraction in \eqref{sect422nineq}. Let us set $m:=2^n-1\in\mathbb{N}^+$ and consider the decomposition
\[
\frac{\left( \Wn(x)-\be_0(x)\right)^{2^n}}{\left( \Wn(x) \right)^{-1+2^n}}=
\left( \Wn(x)-\be_0(x)\right)\cdot\left(1-\frac{\be_0(x)}{\Wn(x)}\right)^{m}.
\]
The first factor is upper estimated by using Lemma \ref{sect42lem47}. As for the second one, notice that
\[
\left(1-\frac{\be_0}{\Wn}\right)'(x)=-\frac{1}{e (\Wn(x)+1)}<0,
\]
hence, for $0<x<e$,
\[
0<1-\frac{\be_0(x)}{\Wn(x)}<\lim_{x\to 0^+}\left(1-\frac{\be_0(x)}{\Wn(x)}\right)=\lim_{x\to 0^+}\left(1-\frac{x/e}{\Wn(x)}\right).
\]
Now \eqref{W00series}---the Taylor expansion of $\Wn$ about the origin, with positive radius of convergence---implies that $\lim_{x\to 0}\frac{\Wn(x)}{x}=1$, so the above limit is $\kappa_2$, completing the proof.
\end{proof}

Theorem \ref{thm47sect42} is illustrated by Figure \ref{fig2region}.

\begin{figure}
\begin{center}
\includegraphics[width=0.66\textwidth]{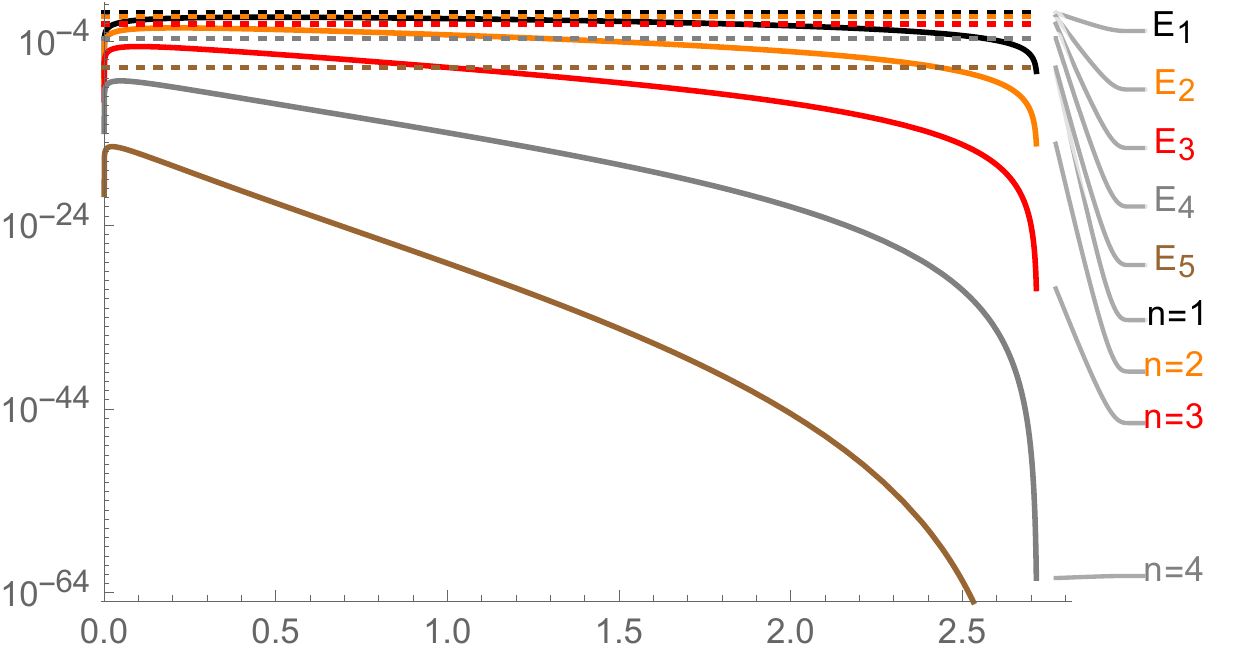}
\caption{A semi-log plot illustrating Theorem \ref{thm47sect42}. The continuous curves correspond to the actual differences $\Wn(x)-\be_n(x)$ for various values of $n$ (and they tend to $0$ as $x$ converges to any of the endpoints of the interval $(0, e)$), whereas the dotted lines $E_n$ depict the uniform estimates $\frac{1}{5}\cdot\kappa_2^{-1+2^n}$ in \eqref{thm47unifest}.\label{fig2region}}
\end{center}
\end{figure}

\subsection{Convergence to $\Wn$ on the interval $(-1/e,0)$} \label{subsection43}

Let us fix any $x\in(-1/e,0)$ in this section. On this interval, we make the following choice for the starting value:
\begin{equation}\label{-1/e<x<0start}
\left\{
\begin{aligned}
  \be_0(x):=  & \frac{e x\ln(1+\sqrt{1+e x})}{\sqrt{1+e x}\,(1+\sqrt{1+e x})}, \\ 
  \be_{n+1}(x):=  &  \frac{\be_n(x)}{1+\be_n(x)}\left(1+\ln\left(\frac{x}{\be_n(x)}\right)\right)\quad (n\in\mathbb{N}). 
\end{aligned}
\right.
\end{equation}

The following lemma gives a two-sided initial estimate of $\Wn$; its proof is given in Section \ref{sect43w0<b0lemmaproof}.

\begin{lem}\label{sect43w0<b0}
For any $-1/e<x<0$ we have
\[
-1<-1+\sqrt{1+e x}<\Wn(x)<\be_0(x)<0.
\]
\end{lem}

\begin{rem}\label{section43rem411}
The choice of the lower bound $-1+\sqrt{1+e x}$ in Lemma \ref{sect43w0<b0} is motivated by the Puiseux expansion of $\W$ about the branch point $-1/e$, while $\be_0(x)$ is the result of a single iteration step of \eqref{IBrecursion} applied to $-1+\sqrt{1+e x}$. 
\end{rem}

The lemma below establishes the monotonicity and boundedness properties of the sequence \eqref{-1/e<x<0start}, and shows that it is well-defined and real-valued. Its proof---found in Secion \ref{lem412lemmaproof}---is analogous to that of Lemma \ref{lem41}.
\begin{lem}\label{lem412}
For any $x\in(-1/e,0)$, the recursion \eqref{-1/e<x<0start}  satisfies 
\[
-1<\Wn(x)<\be_{n+1}(x)<\be_n(x)<0\quad (n\in\Nn).
\]
\end{lem}

The error estimate in Theorem \ref{sect43thm} will be based on the following inequality (cf.~Lemma \ref{sect41lem42}), whose proof is found in Section \ref{lemma413proof}. 
\begin{lem}\label{sect43lem413}
For any $-1/e<x<0$ and $n\in\mathbb{N}$, we have
\begin{equation}\label{sect43lem413ineq}
0<\be_{n+1}(x)-\Wn(x)<\frac{(\be_n(x)-\Wn(x))^2}{-\Wn(x)(1+\be_n(x))}.
\end{equation}
\end{lem}
Regarding the above upper estimate, note that this time the denominator of the fraction in \eqref{sect43lem413ineq} can get arbitrarily close to $0$ near \emph{both} endpoints of the interval $(-1/e,0)$. 

The following two lemmas constitute the final building blocks in the proof of Theorem \ref{sect43thm}, with proofs in Sections \ref{sect43b0w0globmaxlemmaproof} and \ref{highpowerglobmaxlemmaproof}, respectively.

\begin{lem}\label{sect43b0w0globmax}
For any $-1/e<x<0$, we have
\[
0<\be_0(x)-\Wn(x)<\frac{1}{10}.
\]
\end{lem}
\begin{rem} The upper bound $1/10$ in Lemma \ref{sect43b0w0globmax} could be replaced by, say, $0.015$, but the proof of that inequality would require more effort.
\end{rem}

\begin{lem}\label{highpowerglobmax}
For any $-1/e<x<0$, we have
\[
0<\frac{\be_0(x)-\Wn(x)}{-\Wn(x)\sqrt{1+e x}}<\frac{1}{10}.
\]
\end{lem}

The main result of this section is given below, also proving convergence of the recursion \eqref{-1/e<x<0start} on $(-1/e,0)$.

\begin{thm}\label{sect43thm}
For any $n\in\mathbb{N}^+$ and $x\in (-1/e,0)$, the recursion \eqref{-1/e<x<0start} satisfies the uniform estimate
\[
0<\be_n(x)-\Wn(x)<\left(\frac{1}{10}\right)^{2^n}.
\]
\end{thm}
\begin{proof} Due to \eqref{sect43lem413ineq} and Lemmas \ref{sect43w0<b0} and \ref{lem412}, we have
\[
0<\be_{n}(x)-\Wn(x)<\frac{(\be_{n-1}(x)-\Wn(x))^2}{-\Wn(x)(1+\be_{n-1}(x))}<\frac{(\be_{n-1}(x)-\Wn(x))^2}{-\Wn(x)\sqrt{1+e x}},
\]
so, recursively, we get
\[
0<\be_{n}(x)-\Wn(x)<\left(\be_0(x)-\Wn(x)\right)\cdot \left(\frac{\be_0(x)-\Wn(x)}{-\Wn(x)\sqrt{1+e x}}\right)^{-1+2^n}.
\]
Now Lemmas \ref{sect43b0w0globmax} and \ref{highpowerglobmax} finish the proof.
\end{proof}

Theorem \ref{sect43thm} is illustrated by Figure \ref{fig3region}.

\begin{figure}
\begin{center}
\includegraphics[width=0.66\textwidth]{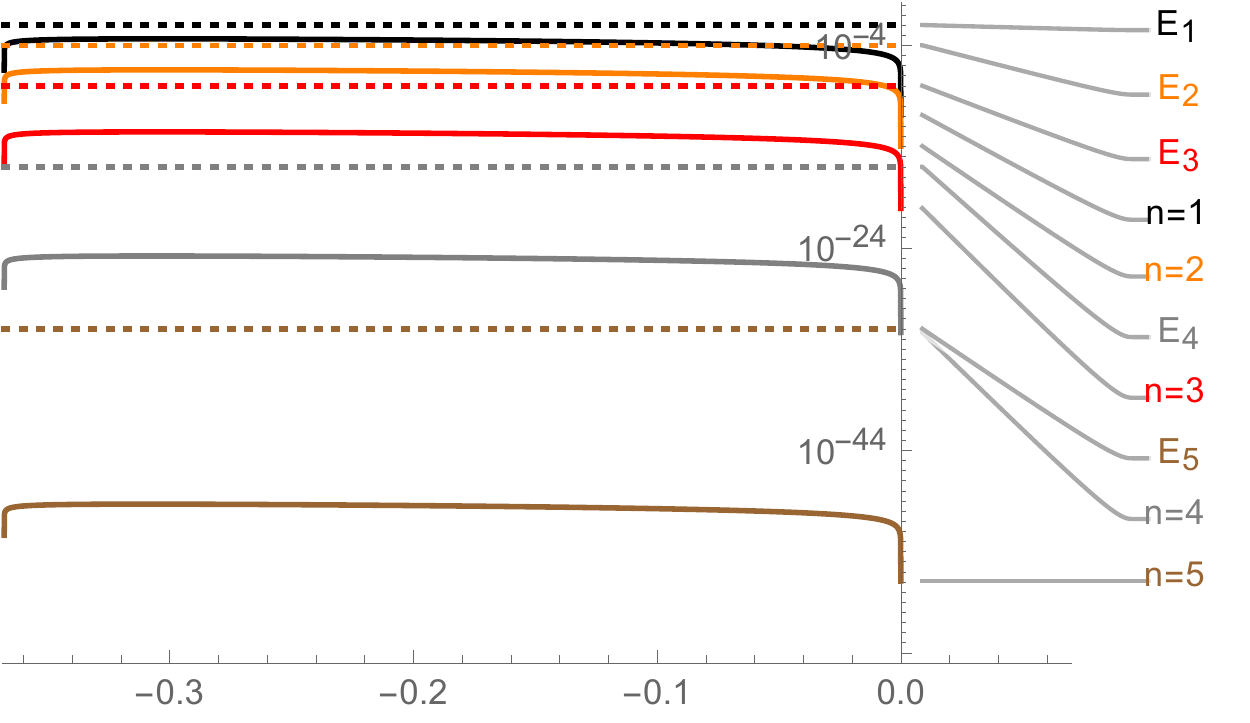}
\caption{A semi-log plot illustrating Theorem \ref{sect43thm}. The continuous curves correspond to the actual differences $\be_n(x)-\Wn(x)$ for various values of $n$ (and they tend to $0$ as $x$ converges to any of the endpoints of the interval $(-1/e, 0)$), whereas the dotted lines $E_n$ depict the uniform estimates $1/10^{2^n}$.\label{fig3region}}
\end{center}
\end{figure}

\begin{rem}
In \cite[Section 6]{johansson}, the first 9950 digits of the quantity $\Wn\left(-\frac{1}{e}+10^{-100}\right)$ near the branch point are computed. We computed $\be_{14}\left(-\frac{1}{e}+10^{-100}\right)$ by using \textit{Mathematica} ($2^{14}>9950$), and found that all the first and last few digits displayed in \cite{johansson} are again in agreement---the computations within two different systems yielded the same result.
\end{rem}

\subsection{Convergence to $\Wm$ on the interval $(-1/e,0)$}\label{subsection44}

In this section we propose suitable starting values for the recursion \eqref{IBrecursion} to converge to $\Wm(x)$ for any $x\in(-1/e,0)$. The convergence is again proved via simple (uniform) error estimates.

Although the statements and proofs are similar to those in Sections \ref{IBconvx>e}--\ref{subsection43}, let us highlight some differences, including
\begin{itemize}
\item the branch $\Wm$ over the bounded interval $(-1/e,0)$ is unbounded---with a branch point at the left endpoint, and a singularity at the right endpoint---hence we will split $(-1/e,0)$ when defining the recursion starting values $\be_0(x)$;
\item when using the bijective reparametrization $x=y e^y$ in the proofs of transcendental inequalities to eliminate $\Wm$, this time $\Wm(y e^y)=y$ will hold for $y<-1$ (cf.~the identity $\Wn(y e^y)=y$ for $-1<y<0$ used earlier);
\item the branch $\Wm$ is strictly decreasing, so instead of \eqref{monotoneequivalence} we now have 
\begin{equation}\label{reversemonotoneequivalence}
\alpha \ \boxed{ \lesseqqgtr} \ \beta\text{\ if and only if } \alpha e^\alpha \ \boxed{ \gtreqqless} \ \beta e^\beta,
\end{equation}
for any $\alpha, \beta\in (-\infty,-1]$.
\end{itemize}
Due to the above reasons, the proofs will be presented in detail.\\

For $x\in(-1/e,0)$, we define the recursion as follows:
\begin{equation}\label{W-1start}
\left\{
\begin{aligned}
  \be_0(x):=  & -1-\sqrt{2} \sqrt{1+e x}\quad\quad\quad \text{ for } -1/e<x\le-1/4, \\ 
  \be_0(x):=  & \ln (-x)-\ln (-\ln (-x))\quad\quad \text{ for } -1/4<x< 0, \\ 
  \be_{n+1}(x):=  &  \frac{\be_n(x)}{1+\be_n(x)}\left(1+\ln\left(\frac{x}{\be_n(x)}\right)\right)\quad\quad (n\in\mathbb{N}). 
\end{aligned}
\right.
\end{equation}

\begin{rem}
\emph{(i)} The point $-1/4$ to split the interval $(-1/e,0)$ in the definition of $\be_0$ in \eqref{W-1start} is somewhat arbitrary; it has been chosen to make the constants in the estimates of this section  
simple, small positive numbers.\\
\emph{(ii)} With the above definition, $\be_0$ is a piecewise continuous function. In fact, it is possible to construct a function that is continuous over the \emph{whole} interval $(-1/e,0)$ and approximates $\Wm$ so well that all the lemmas and the theorem below would remain true (with slightly different constants, of course). The choice $\widetilde{\be}_0(x):=\ln (-x)-\ln (-\ln (-x))$, for example, would \emph{not} be an appropriate one on the interval $(-1/e,0)$, as it would result in some singular estimates near $x=-1/e$. One suitable choice for the starting value of \eqref{W-1start} could be
\[
\widetilde{\be}_0(x):=\frac{e x \ln \left(1-\sqrt{1+e x}\right) \ln \left(\frac{1+e x-\sqrt{1+e x}}{
   \ln \left(1-\sqrt{1+e x}\right)}\right)}{1+e x-\sqrt{1+e x}+e x \ln \left(1-\sqrt{1+e x}\right)} \quad\quad (x\in (-1/e,0)),
\]
but with this formula the proofs of the estimates would become more involved. We remark that the difference $\Wm-\widetilde{\be}_0$ is strictly increasing and satisfies 
\[
0<\Wm(x)-\widetilde{\be}_0(x)<\lim_{x\to 0^-}\left(\Wm(x)-\widetilde{\be}_0(x)\right)=\ln(2)-\frac{1}{2}\approx 0.193
\]
for any $x\in (-1/e,0)$. The expression for $\widetilde{\be}_0(x)$ has been obtained by taking \emph{two} iteration steps with \eqref{IBrecursion} started from $-1-\sqrt{1+e x}$ (cf.~Remark \ref{section43rem411}).\\
\emph{(iii)} Regarding the factor $\sqrt{2}$ in the definition of $\be_0$ in  \eqref{W-1start}, it directly appears in the Puiseux expansion of $\W$ about $x=-1/e$, and it gives a better approximation for $\Wm$ close to $-1/e$. However, the constant $\sqrt{2}$ was \emph{not} included in the starting value of the recursion \eqref{-1/e<x<0start}, because this way that $\be_0$ yields an overall better estimate  for $\Wn$ on $(-1/e,0)$.\\
\emph{(iv)} The choice for the other starting value in \eqref{W-1start} is motivated by the estimate 
\eqref{W-1bound}. This estimate appears in \cite{faris} (but by using a different---equivalent---parametrization).\\
\emph{(v)} As we will see, the sequence $\be_n$ is only monotone for $n\in\mathbb{N}^+$ (and not for $n\in\mathbb{N}$). Again, this is a consequence of the trade-off between simple proofs and 
 good uniform error estimates. 
\end{rem}

The first lemma estimates the initial difference; its proof is found in Section \ref{sect44lem421lemmaproof}.

\begin{lem}\label{sect44lem421}
For any $x\in(-1/e,0)$, the starting value in \eqref{W-1start} satisfies the estimates
\begin{equation}\label{sect44lem21est}
0<\be_0(x)-\Wm(x)<1/2.
\end{equation}
\end{lem} 

The well-definedness and monotonicity properties of the sequence $\be_n$, and the inductive part of the error estimates are summarized next. The proof of the lemma is given is Section \ref{sect44lem419lemmaproof}.

\begin{lem}\label{sect44lem419}
For any $x\in(-1/e,0)$ and $n\in\mathbb{N}^+$, the recursion \eqref{W-1start} is well-defined, real-valued, and satisfies the following:
\begin{equation}\label{sect44lem419mon}
\be_n(x)<\be_{n+1}(x)<\Wm(x)<\be_0(x)<-1,
\end{equation}
and
\begin{equation}\label{sect44lem419est}
0<\Wm(x)-\be_{n}(x)<\left(\be_0(x)-\Wm(x)\right)\cdot \left(\frac{\be_0(x)-\Wm(x)}{|\Wm(x)|\cdot |1+\be_0(x)|}\right)^{-1+2^n}.
\end{equation}
\end{lem} 
Regarding the upper estimate \eqref{sect44lem419est}, note that this time its denominator can get arbitrarily close to $0$ near the left endpoint of the interval $(-1/e,0)$, and both terms in its numerator are singular as $x\to 0^-$. The following lemma yields suitable upper estimates of this fraction. Its proof is found in Section \ref{sect44lem422lemmaproof}.

\begin{lem}\label{sect44lem422}
For any $x\in(-1/e,0)$, the starting value in \eqref{W-1start} satisfies the estimates
\begin{equation}\label{sect44lem22est}
0<\frac{\be_0(x)-\Wm(x)}{|\Wm(x)|\cdot |1+\be_0(x)|}<\frac{1}{2}.
\end{equation}
Moreover, for $-1/4<x< 0$ we also have
\begin{equation}\label{sect44lem22estsharper}
0<\frac{\be_0(x)-\Wm(x)}{|\Wm(x)|\cdot |1+\be_0(x)|}<\frac{1/2}{|\ln (-x)-\ln (-\ln (-x)) |\cdot | 1+\ln (-x)-\ln (-\ln (-x))|}.
\end{equation}
\end{lem} 

Summarizing the above, we have the following main result.
\begin{thm}\label{W-1thm}
For any $-1/e<x<0$ and $n\in\mathbb{N}^+$, the recursion \eqref{W-1start} satisfies
\[
0<\Wm(x)-\be_{n}(x)< \left(\frac{1}{2}\right)^{2^n}.
\]
In particular, for $-1/4<x< 0$, the sharper estimate
\[
\Wm(x)-\be_{n}(x)< \left(\frac{1}{2}\right)^{2^n} \left(\frac{1}{|\ln (-x)-\ln (-\ln (-x)) |\cdot | 1+\ln (-x)-\ln (-\ln (-x))|}\right)^{-1+2^n}
\]
also holds.
\end{thm}
\begin{proof} The proof directly follows by combining Lemma \ref{sect44lem419} with Lemmas \ref{sect44lem421} and \ref{sect44lem422}.
\end{proof}

Theorem \ref{W-1thm} is illustrated by Figure \ref{fig4region}.

\begin{figure}
\begin{subfigure}{.5\textwidth}
  \centering
  \includegraphics[width=.9\linewidth]{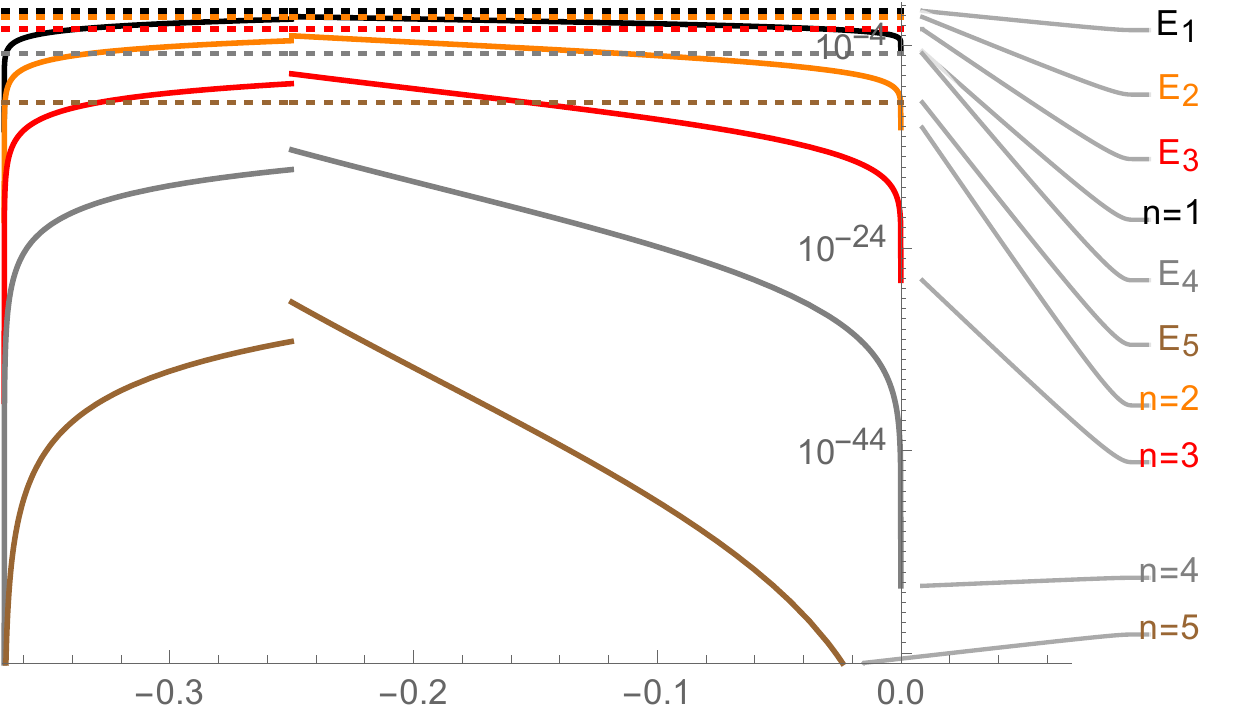}
  \caption{}
\end{subfigure}%
\begin{subfigure}{.5\textwidth}
  \centering
  \includegraphics[width=.9\linewidth]{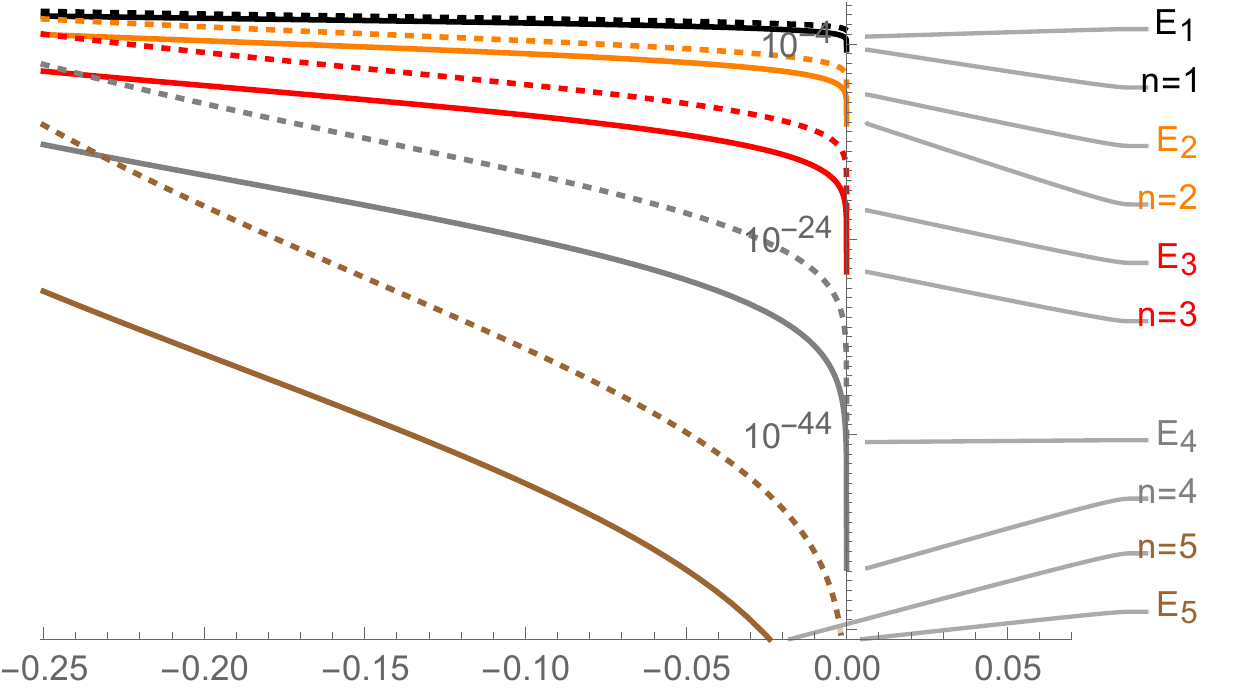}
  \caption{}
\end{subfigure}
\caption{Semi-log plots illustrating Theorem \ref{W-1thm}. The piecewise continuous curves in figure (a) (reflecting the different definitions for $x\in(-1/e,-1/4]$ and for $x\in(-1/4,0)$ in \eqref{W-1start}) correspond to the actual differences $\Wm(x)-\be_n(x)$ for various values of $n$, whereas the dotted lines $E_n$ depict the uniform estimates $1/2^{2^n}$. In figure (b), the dotted curves $E_n$ correspond to the right-hand side of the ``sharper estimate'' in the theorem for $x\in(-1/4,0)$.}
\label{fig4region}
\end{figure}

\section{An application: the non-trivial positive real solutions of $x^y=y^x$}\label{sectionxyyx}

In this section, the function $y: (1,+\infty)\to (1,+\infty)$ denotes the unique smooth
solution to the implicit equation 
\begin{equation}\label{ydef}
x^{y(x)}=(y(x))^x
\end{equation} with $y(x)\ne x$ for $x\ne e$
 (and then we necessarily have $y(e)=e$).
There are several papers in the literature on the solutions to the commutative
equation of exponentiation $x^y=y^x$, first considered by D.~Bernoulli, C.~Goldbach, and L.~Euler; see the survey \cite{loczi}.
In connection with the function $y$ in \eqref{ydef}, the following question and conjecture were posed in \cite{gofen} based on numerical observations.
\begin{quest}\label{question21}
Does the function $x\mapsto x^{y(x)}$ have an asymptote at $+\infty$?
\end{quest}
As for the function $y$ itself, Figure \ref{fig1} suggests that it can be well approximated by the unique hyperbola with
vertical asymptote $x=1$, horizontal asymptote $y=1$ and slope of the tangent at $x=e$ equal to $-1$.
\begin{figure}
\begin{center}
\includegraphics[width=0.4\textwidth]{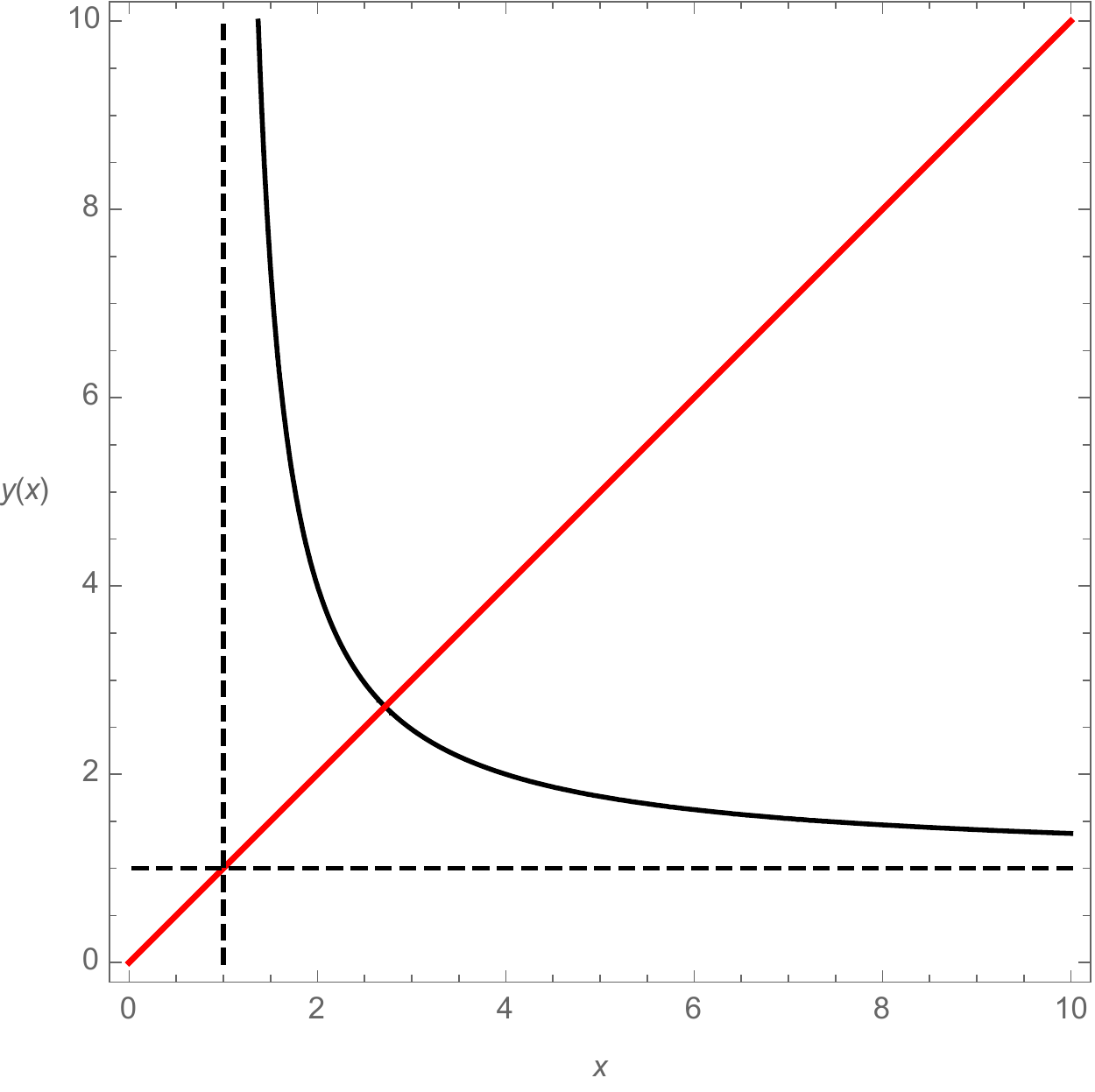}
\caption{The continuous black curve shows the non-trivial positive solution curve of
the equation $x^y=y^x$, while the red line represents the trivial solution set $y=x$.
\label{fig1}}
\end{center}
\end{figure}
\begin{conj}\label{conjecture21}
Prove that for any $x\in(1,+\infty)\setminus\{e\}$, we have 
\[
y(x)>1+\frac{(e-1)^2}{x-1}.
\]
\end{conj}
In this section we show that the answer to Question \ref{question21} is negative, but Conjecture \ref{conjecture21} is true.

Due to symmetry, it is enough
to verify Conjecture \ref{conjecture21} for $x>e$, 
and
we can clearly restrict $y$ to the same interval to answer Question \ref{question21}.
It is known \cite{loczi} that for $x>e$ 
\begin{equation}\label{yrep}
y(x)=-\frac{x}{\ln (x)}\cdot \Wn\left(-\frac{\ln (x)}{x}\right). 
\end{equation}

\begin{rem}
It can be checked that the branch \eqref{yrep} for $x>e$ can also be written as
$y(x)=\exp\left(-\Wn\left(-\frac{\ln (x)}{x}\right)\right)$. As for the branch defined for $1<x<e$, it can be represented by $\Wm$ as $y(x)=\exp\left(-\Wm\left(-\frac{\ln (x)}{x}\right)\right)$.
\end{rem}

Let us present a simple lemma first, derived from \eqref{W00series}. Its proof is found in Section \ref{section33}. 
\begin{lem}\label{smallWestimate}
For any $x\in(-1/4,0)$, we have
\[
x-4 x^2<\frac{1}{2} \left(\sqrt{1+4 x}-1\right)<\Wn(x)<\sqrt{1+2 x}-1<x-\frac{x^2}{2}<x.
\]
\end{lem}

Now if there is an asymptote to $x\mapsto x^{y(x)}$ at $+\infty$ of the form $x\mapsto a x+b$
with some $a, b\in\mathbb{R}$, then
\[
a=\lim_{x\to +\infty}\frac{x^{y(x)}}{x}\quad\text{ and }\quad
b=\lim_{x\to +\infty}(x^{y(x)}-a x).
\] 
Due to \eqref{yrep} we have
${x^{y(x)}}/{x}=\exp\left({-x\, \Wn\left(-\frac{\ln (x)}{x}\right)}-\ln(x)\right)$, so 
$a=e^0=1$, because, by Lemma \ref{smallWestimate}, for sufficiently large $x$ (e.g., $x>9$ works)  
\[
0=-x\left(-\frac{\ln (x)}{x}\right)-\ln (x)<{-x\, \Wn\left(-\frac{\ln (x)}{x}\right)}-\ln(x)<
\]
\[
-x \left(-\frac{\ln (x)}{x}-4 \left(-\frac{\ln (x)}{x}\right)^2\right)-\ln (x)=
\frac{4 \ln ^2(x)}{x}.
\]
As for the quantity $b$, we use Lemma \ref{smallWestimate} again to get
\[
x^{y(x)}-x=\exp\left({-x\, \Wn\left(-\frac{\ln (x)}{x}\right)}\right)-x>
\]
\[
\exp \left(-x \left(-\frac{\ln (x)}{x}-\frac{1}{2} \left(-\frac{\ln (x)}{x}\right)^2\right)\right)-x=
\]
\[
x \left(\exp\left({\frac{\ln^2(x)}{2 x}}\right)-1\right)>
x \left(1+{\frac{\ln^2(x)}{2 x}}-1\right)=\frac{\ln ^2(x)}{2}.
\]
This estimate shows that $b$ cannot be finite, hence  there is no asymptote 
to the function $x\mapsto x^{y(x)}$ at $+\infty$.
\begin{rem} A slightly stronger statement is
\[
\lim_{x\to +\infty} \frac{\exp\left({-x\, \Wn\left(-\frac{\ln (x)}{x}\right)}\right)-x}{\ln^2(x)}=1,
\]
which we present here without proof. 
\end{rem}
Now we turn to Conjecture \ref{conjecture21}. By using \eqref{yrep} again we see that 
the claim is equivalent
to 
\[
\alpha(x):=\Wn\left(-\frac{\ln (x)}{x}\right)<-\frac{\ln (x)}{x}\left(1+\frac{(e-1)^2}{x-1}\right)=:\beta(x)
\]
for $x>e$. 
It is easy to check that $-1<\al(x)$ for $x>e$. The next lemma (to be proved in Section \ref{section31}) shows similarly that $-1<\beta$ on $(e,+\infty)$. 
\begin{lem}\label{betalemma}
For any $x>e$, one has $-1<\beta(x)$.
\end{lem}
Now, by using \eqref{monotoneequivalence},
$\alpha(x)<\beta(x)$ for $x>e$ is equivalent to  
\[
-\frac{\ln (x)}{x}<-\left(1+\frac{(e-1)^2}{x-1}\right) \frac{\ln (x)}{x} \,\exp \left(-\left(1+\frac{(e-1)^2}{x-1}\right) \frac{\ln (x)}{x}\right).
\]
It is convenient to rewrite the above inequality into the following form.
\begin{lem}\label{lemma25}
For any $x>e$
\begin{equation}\label{lemma25LHS}
\frac{\left(x+e^2-2 e\right) }{(x-1)}\cdot \frac{\ln(x)}{x}+\ln \left(1-\frac{(e-1)^2}{x+e^2-2 e}\right)>0.
\end{equation}
\end{lem}
The proof of this lemma---given in Section \ref{section32}---finishes the proof of Conjecture 
\ref{conjecture21}.


\appendix 

\section{Appendix: the proofs of the lemmas and theorems}\label{proofsection}

\subsection{The proof of Lemma \ref{lemmaWsimple}}\label{sectionWsimple}

The function $\Wn$ is strictly increasing on $[e,+\infty)$ and $\Wn(e)=1$, so by using \eqref{monotoneequivalence} we see that $\Wn(x)<L_1(x)$ for $x>e$ is equivalent to
$x= \Wn(x) e^{\Wn(x)}<\ln(x) e^{\ln(x)}=x \ln(x)$. Similarly, for the lower estimate,
\eqref{ineq1} says that $L_1(x)-L_2(x)>0$ for $x\ge e$, hence by \eqref{monotoneequivalence} 
again,
$L_1(x)-L_2(x)<\Wn(x)$ is equivalent to 
\[x\left(1-\frac{\ln(\ln(x))}{\ln(x)}\right)=(L_1(x)-L_2(x)) e^{L_1(x)-L_2(x)}<\Wn(x) e^{\Wn(x)}=x,\]
but $1-\frac{\ln(\ln(x))}{\ln(x)}<1$ for $x\in(e,+\infty)$, so the proof is complete.

\subsection{The proof of Lemma \ref{9007lemma}}\label{lemma9007lemmaproofsection}

First we present a two-sided estimate of the expression $y^{1/y}$.
\begin{lem}\label{logy/yserieslemma}
For any $1<y$, $2\le N\in\mathbb{N}$ and $1\le N^*\in\mathbb{N}$ we have
\[
\sum_{k=0}^{N^*}\frac{1}{k!}\left(\frac{\ln(y)}{y}\right)^k<y^{1/y}<
\left(\sum_{k=0}^{N-1}\frac{1}{k!}\left(\frac{\ln(y)}{y}\right)^k\right)+\frac{4/3}{N!}\left(\frac{\ln(y)}{y}\right)^N.
\]
\end{lem}
\begin{proof}
Since $y^{1/y}=\exp\left(\frac{\ln(y)}{y}\right)=\displaystyle\sum_{k=0}^{+\infty}\frac{1}{k!}\left(\frac{\ln(y)}{y}\right)^k$ and $y>1$, the lower estimate is verified. As for the upper estimate, one has
\[
y^{1/y}-\sum_{k=0}^{N-1}\frac{1}{k!}\left(\frac{\ln(y)}{y}\right)^k=\sum_{k=N}^{+\infty}\frac{1}{k!}\left(\frac{\ln(y)}{y}\right)^k=\frac{1}{N!}\left(\frac{\ln(y)}{y}\right)^N \ \sum_{k=0}^{+\infty}\frac{1}{(N+k)!/N!}\left(\frac{\ln(y)}{y}\right)^k,
\]
hence it is enough to show that the rightmost sum above is at most $4/3$. Now we take into account that $\ln(y)/{y}\in(0,1/e]$ for $y>1$, and $N\ge 2$, so 
\[
\sum_{k=0}^{+\infty}\frac{1}{(N+k)!/N!}\left(\frac{\ln(y)}{y}\right)^k\le \sum_{k=0}^{+\infty}\frac{1}{(2+k)!/2!}\left(\frac{1}{e}\right)^k=2e \left(e^{1+1/e}-e-1\right)<\frac{4}{3}.
\]
\end{proof}
Now we can start the actual proof of Lemma \ref{9007lemma}. Clearly $f_1(1)=0$. By using
Lemma \ref{1lemma}, Lemma \ref{logy/yserieslemma} with $N^*:=1$, and 
Lemma \ref{72lemmasimple}, for $1<y\le 7/2$ we have
\[
f_1(y)> \left(1+\frac{\ln (y)}{y}\right) \left(y^2-y \ln (y)+\ln(y)\right)-y^2=
\frac{\ln (y)}{y} (y+(1-y) \ln(y))>0.
\]
On the other hand, by Lemma \ref{logy/yserieslemma} with $N:=2$ we see for $y>1$ that
\[
f_1(y)<\left(1+\frac{\ln (y)}{y}+\frac{2 \ln ^2(y)}{3 y^2}\right) \left(y^2-y \ln (y)+\ln(y)\right)-y^2=
\]
\[
\frac{2 \ln^3(y)}{3 y^2}-\frac{2 \ln^3(y)}{3 y}+\frac{\ln^2(y)}{y}-\frac{1}{3} (\ln(y)-3) \ln(y),
\]
so $\displaystyle \lim_{+\infty} f_1=-\infty$. Therefore, the proof of Lemma \ref{9007lemma} will
be complete as soon as we have shown that $f'_1<0$ on $(7/2,+\infty)$. We have
\[
y^2 f'_1(y)=\left[2 y^3-2 y^2 \ln(y)+y(\ln^2(y)-\ln(y)+1)-\ln^2(y)+\ln(y)\right] y^{\frac{1}{y}}-2 y^3.
\]
By Lemma \ref{72lemma}, the expression in $[\ldots]$ is positive, so  
$y^{\frac{1}{y}}$ can be
estimated from above by 
Lemma \ref{logy/yserieslemma} with $N:=3$ and for any $y>7/2$ to get
$y^2 f'_1(y)<\frac{P_{\ref{lemma9007lemmaproofsection}}(y,\ln(y))}{18y^3}$,
where
\[
P_{\ref{lemma9007lemmaproofsection}}(y,z):=(4 y-4) z^5+\left(y^2-13 y+4\right) z^4+\left(8 y^3-27 y^2+13y\right) z^3+
\]
\[
\left(-36 y^3+27 y^2\right) z^2+\left(-18 y^4+36 y^3\right) z+18 y^4.
\]
We know that $y>7/2$, so $\ln(y)>5/4$,  and $\ln (y)<\frac{5 }{6} \sqrt{y}$ by Lemma \ref{56lemma}.
We finish the proof of  Lemma \ref{9007lemma} by showing that
$P_{\ref{lemma9007lemmaproofsection}}(y,z)<0$ for  $(y,z)\in {\mathcal{S}}_{\ref{lemma9007lemmaproofsection}}$, 
where
${\mathcal{S}}_{\ref{lemma9007lemmaproofsection}}:=\left\{ (y,z)\in\mathbb{R}^2 : y>\frac{7}{2} \text{  and  } \frac{5}{4}<z
<\frac{5}{6}\sqrt{y}  \right\}$. Notice that we also have
\[
{\mathcal{S}}_{\ref{lemma9007lemmaproofsection}}=\left\{(y,z)\in\mathbb{R}^2 :   \frac{5}{4}<z\leq \frac{5}{6} \sqrt{\frac{7}{2}} \text{ and } y>\frac{7}{2}\right\}\cup\]
\[
 \left\{(y,z)\in\mathbb{R}^2 :  z>\frac{5}{6} \sqrt{\frac{7}{2}} \text{ and  } y>\frac{36}{25}z^2\right\}.
\]
We use repeated (partial) differentiation to decrease the degree of $P_{\ref{lemma9007lemmaproofsection}}$. It is elementary to see that 
\[
\partial_1 (\partial_1 \partial_1 \partial_2 P_{\ref{lemma9007lemmaproofsection}})(y,z)=-432 y+144 z^2-432 z+216
\]
is negative for $(y,z)\in {\mathcal{S}}_{\ref{lemma9007lemmaproofsection}}$. Moreover, 
\[
\partial_1 \partial_1 \partial_2 P_{\ref{lemma9007lemmaproofsection}}(7/2,z)=8 z^3+342 z^2-1404 z-1890<0 
\]
for $\frac{5}{4}<z\leq \frac{5}{6} \sqrt{\frac{7}{2}}$ and $\partial_1 \partial_1 \partial_2 P_{\ref{lemma9007lemmaproofsection}}(36z^2/25,z)= $
\[
-\frac{2}{625} z \left(75168 z^3+191900 z^2-46575 z-33750\right)<0 
\]
for $z>\frac{5}{6} \sqrt{\frac{7}{2}}$, therefore $\partial_1 (\partial_1 \partial_2 P_{\ref{lemma9007lemmaproofsection}})(y,z)<0$ for 
$(y,z)\in {\mathcal{S}}_{\ref{lemma9007lemmaproofsection}}$. But 
\[
\partial_1 \partial_2 P_{\ref{lemma9007lemmaproofsection}}(7/2,z)=20 z^4-24 z^3+354 z^2-2268 z-1764<0,
\]
for $\frac{5}{4}<z\leq \frac{5}{6} \sqrt{\frac{7}{2}}$, and $\partial_1 \partial_2 P_{\ref{lemma9007lemmaproofsection}}(36z^2/25,z)=$
\[
\frac{z^2}{15625} \left(-1026432 z^4-6818400 z^3+166700 z^2+1617500 z+609375\right)<0
\]
for $z>\frac{5}{6} \sqrt{\frac{7}{2}}$, so $\partial_1 (\partial_2 P_{\ref{lemma9007lemmaproofsection}})(y,z)<0$ for 
$(y,z)\in {\mathcal{S}}_{\ref{lemma9007lemmaproofsection}}$. Analogously, 
\[
\partial_2 P_{\ref{lemma9007lemmaproofsection}}(7/2,z)=50 z^4-117 z^3+\frac{693 z^2}{4}-\frac{4851 z}{2}-\frac{9261}{8}<0
\]
for $\frac{5}{4}<z\leq \frac{5}{6} \sqrt{\frac{7}{2}}$, and (repeated differentiation with respect to $z$ shows that) 
\[
\partial_2 P_{\ref{lemma9007lemmaproofsection}}(36z^2/25,z)=-\frac{8 z^3 }{390625}\times
\]
\[
 \left(279936 z^5+10092600 z^4+1546200 z^3-1811250 z^2-1765625 z-781250\right)<0
\]
for $z>\frac{5}{6} \sqrt{\frac{7}{2}}$, hence $\partial_2 P_{\ref{lemma9007lemmaproofsection}}(y,z)<0$ for 
$(y,z)\in {\mathcal{S}}_{\ref{lemma9007lemmaproofsection}}$. But
\[
P_{\ref{lemma9007lemmaproofsection}}(y,5/4)=\frac{1}{256} \left(-1152 y^4+1120 y^3-2075 y^2+1500 y-625\right)<0
\]
for $y>7/2$, so $P_{\ref{lemma9007lemmaproofsection}}(y,z)<0$ for $(y,z)\in {\mathcal{S}}_{\ref{lemma9007lemmaproofsection}}$. 
The proof of Lemma \ref{9007lemma} is complete. 

\begin{rem} The above simple proof of the negativity of the polynomial $P_{\ref{lemma9007lemmaproofsection}}$ would
break down if the constant $5/6$ in Lemma \ref{56lemma} were replaced by, say, $1$.
\end{rem}

\subsection{The proof of Theorem \ref{maintheorem}}\label{section34}

We know that 
$
L_1(e)-L_2(e)+\frac{L_2(e)}{L_1(e)}=\Wn(e)=1,
$
and on the interval $(e,+\infty)$  both  $\Wn>1$ and $L_1-L_2+\frac{L_2}{L_1}>0$ hold (by \eqref{ineq1} in Lemma \ref{elementarylemma}). Hence, by using \eqref{monotoneequivalence} we have for any $x>e$ that
\[
L_1(x)-L_2(x)+\frac{L_2(x)}{L_1(x)}\ \boxed{ \lesseqqgtr}\ \Wn(x)
\]
is equivalent to
\[
\left(L_1(x)-L_2(x)+\frac{L_2(x)}{L_1(x)}\right) \exp\left(L_1(x)-L_2(x)+\frac{L_2(x)}{L_1(x)}\right) \ \boxed{ \lesseqqgtr}\ 
\]
\[
\Wn(x) \exp(\Wn(x))=x,
\]
that is, to
\[
\frac{x}{{\ln^2(x)}} \left(y^{1/y} \left(y^2-y \ln (y)+\ln(y)\right)-y^2\right) \ \boxed{ \lesseqqgtr}\ 0
\]
with $y:=\ln(x)$, where $\boxed{ \lesseqqgtr}$ stands for either ``$<$'', or ``$>$'', or ``$=$''. Lemma \ref{9007lemma} then proves the statement in the first sentence of Theorem
\ref{maintheorem}. 

Now by using the preliminary lower bound 
$L_1-L_2+\frac{L_2}{L_1}<\Wn$ 
we have just obtained on $(x^*,+\infty)$, we prove the upper bound \eqref{maintheoremupperestimate}. To this end, we notice that the function $-L_1$  in the identity \eqref{functional} is strictly decreasing, hence
\[
\Wn=L_1-L_1\circ\Wn<L_1-L_1\circ\left(L_1-L_2+\frac{L_2}{L_1}\right)
\]
holds on $(x^*,+\infty)$. By increasing by the right-hand side, see Lemma \ref{thm23auxupperest}, the proof of \eqref{maintheoremupperestimate} is complete.

To finish the proof of Theorem \ref{maintheorem}, we now verify the lower estimate \eqref{lowboundimproved}. Clearly, the inequality $\boxed{<}$ in \eqref{lowboundimproved} is equivalent to $0<(\ln (x)-3)\ln(\ln (x)) -2 \ln (x)$, and this inequality is true on the interval $(x^{**},+\infty)$ due to  Lemma \ref{13.26lemma}. (The proof of Lemma \ref{13.26lemma} is a simple convexity argument: one has $f_2(1)<0$, $\displaystyle \lim_{+\infty}f_2=+\infty$, and $f_2''>0$ on $[1,+\infty)$.)

Finally, we prove the inequality $<$ in \eqref{lowboundimproved} by using the same idea as in the proof of \eqref{maintheoremupperestimate}. The function $-L_1$  in the identity \eqref{functional} is strictly decreasing, so due to \eqref{maintheoremupperestimate} itself, we have
\[
\Wn=L_1-L_1\circ\Wn>L_1-L_1\circ\left(L_1-L_2+\frac{L_2}{L_1}+\frac{(L_2-2) L_2}{2 L_1^2}+
\frac{L_2^3}{L_1^3}\right).
\]
By decreasing by the right-hand side, see Lemma \ref{thm23auxlowerest}, the proof of Theorem \ref{maintheorem} is complete.

\subsection{The proof of Lemma \ref{sandwichlemma}}\label{sandwichlemmaproof}

Let us fix an arbitrary $x>e$. Inequality \eqref{lambdatrivialbounds} for $n=0$ is the elementary
chain $1<\ln(x)<\frac{x}{e}$. We proceed by induction. Suppose
that we have $1<\lambda_n(x)<\frac{x}{e}$ for some $n\in\mathbb{N}$. Then
\[
\lambda_{n+1}(x)=\ln(x)-\ln(\lambda_{n}(x))\in \left( \ln(x)-\ln\left(\frac{x}{e}\right),
 \ln(x)-\ln(1)\right)=
\]
\[
(1,\ln(x))\subset \left(1,\frac{x}{e}\right),
\]
so the induction is complete. This inductive argument shows in particular that the numbers
$\lambda_n(x)$ are real. 

As for \eqref{lambdasandwich}, we prove it on
$(e,+\infty)$ again by induction. The starting step, $\lambda_{1}< \Wn< \lambda_{0}$, is just 
Lemma \ref{lemmaWsimple}. So if we have $\lambda_{2n+1}< \Wn<\lambda_{2n}$
for some $n\in\mathbb{N}$, then by $0<1<\lambda_{2n+1}$ we get
$L_1\circ\lambda_{2n+1}<L_1\circ \Wn<L_1\circ \lambda_{2n}$, that is, 
$L_1-L_1\circ\lambda_{2n+1}>L_1-L_1\circ W>L_1-L_1\circ \lambda_{2n}$, being
the same as $\lambda_{2n+2}>\Wn>\lambda_{2n+1}$
by the functional relation \eqref{functional} and the definition of the recursive sequence \eqref{lambdarecursiondef}. By repeating these manipulations, we get
$\lambda_{2n+3}<\Wn<\lambda_{2n+2}$, so the induction and the proof are complete.

\subsection{The proof of Theorem \ref{lambdatheorem}}\label{sectionlambdatheoremproof}

First we prove the following stronger statement by induction. The constant $\widetilde{x}^*$ 
is defined as $\widetilde{x}^*:=e^{\sqrt{2 \ln (2)}} \sqrt{2 \ln (2)}\in(3.82,3.83)$.
\begin{lem}\label{xtildestarlemma}
For any  $x\ge \widetilde{x}^*$ and $n\in\mathbb{N}$
one has
\begin{equation}\label{lemma53formula}
\lambda_{2n}(x)-\Wn(x)\le (2\ln(2))^n \cdot \frac{\ln(x)-\Wn(x)}{\Wn^{\,2n}(x)}.
\end{equation}
\end{lem}
\begin{proof} We fix any $x\ge \widetilde{x}^*$ in the proof. 
For $n=0$ the claim is trivial, so the induction can be started. 
Suppose that we have already proved  \eqref{lemma53formula} for some $n\in\mathbb{N}$, that is,
we have $\lambda_{2n}(x)\le \Wn(x)\left(1 + (2\ln(2))^n \cdot \frac{\ln(x)-\Wn(x)}{\Wn^{2n+1}(x)}\right)$. 
By Lemma \ref{sandwichlemma},
$1\le \lambda_{2n}(x)$, so we can take logarithms to get
\[
\ln\left(\lambda_{2n}(x)\right)\le \ln(\Wn(x))+\ln\left(1 + (2\ln(2))^n \cdot \frac{\ln(x)-\Wn(x)}{\Wn^{2n+1}(x)}\right).
\]
Now $\widetilde{z}:=(2\ln(2))^n\cdot \frac{\ln(x)-\Wn(x)}{\Wn^{2n+1}(x)}\ge 0$ by Lemma \ref{lemmaWsimple}, 
hence \eqref{functional} and $\ln(1+\widetilde{z})\le \widetilde{z}$ (for $\widetilde{z}>-1$) yield 
$\ln\left(\lambda_{2n}(x)\right)\le \ln(x)-\Wn(x)+(2\ln(2))^n \cdot \frac{\ln(x)-\Wn(x)}{\Wn^{2n+1}(x)}$,
that is
\[
\Wn(x)-\lambda_{2n+1}(x)=\Wn(x)-(\ln(x)-\ln\left(\lambda_{2n}(x)\right))\le (2\ln(2))^n \cdot \frac{\ln(x)-\Wn(x)}{\Wn^{2n+1}(x)}.
\]
By rearranging this we obtain
\[
\Wn(x)\left[ 1-\left(\frac{2\ln(2)}{\Wn^2(x)}\right)^n\cdot \frac{\ln(x)-\Wn(x)}{\Wn^2(x)}\right]\le \lambda_{2n+1}(x).
\]
The assumption $x\ge \widetilde{x}^*$ guarantees that $\frac{2\ln(2)}{\Wn^2(x)}\in(0,1]$, and Lemma \ref{Winductiveestimatelemma} with $m=2$ that $0\le \frac{\ln(x)-\Wn(x)}{\Wn^2(x)}\le \frac{1}{2}$, so the expression in $[\ldots]$ above is positive.  Hence by taking logarithms and using \eqref{functional} again we get
\[
\ln(x)-\Wn(x)+\ln\left( 1-\left(\frac{2\ln(2)}{\Wn^2(x)}\right)^n\cdot \frac{\ln(x)-\Wn(x)}{\Wn^2(x)}\right)\le \ln(\lambda_{2n+1}(x)).
\]
We now use Lemma \ref{2ln2lemma} with $z:=\left(\frac{2\ln(2)}{\Wn^2(x)}\right)^n\cdot \frac{\ln(x)-\Wn(x)}{\Wn^2(x)}\cdot 2\ln(2)\in (0,\ln(2)]$ to decrease the left-hand side and have
\[
\ln(x)-\Wn(x)-\left(\frac{2\ln(2)}{\Wn^2(x)}\right)^n\cdot \frac{\ln(x)-\Wn(x)}{\Wn^2(x)}\cdot 2\ln(2)\le \ln(\lambda_{2n+1}(x)).
\]
Thus 
\[
\lambda_{2n+2}(x)-\Wn(x)=\ln(x)-\ln(\lambda_{2n+1}(x))-\Wn(x)\le \left(2\ln(2)\right)^{n+1}\cdot \frac{\ln(x)-\Wn(x)}{\Wn^{2n+2}(x)},
\]
and the induction is complete.
\end{proof}

To finish the proof of Theorem \ref{lambdatheorem}, we notice, by using Lemma
\ref{lemmaWsimple}, that
\[
(2\ln(2))^n \cdot\frac{\ln(x)-\Wn(x)}{\Wn^{2n}(x)}\le 
(2\ln(2))^n \cdot\frac{\ln(\ln(x))}{(\ln(x)-\ln(\ln(x)))^{2n}}.
\]
By \eqref{ineq1}, the rightmost denominator here is positive, and the restriction $x>x^{***}$ 
guarantees that $\frac{(2\ln(2))^n}{(\ln(x)-\ln(\ln(x)))^{2n}}$ also converges to $0$ as $n\to +\infty$.
The lower estimate $0< \lambda_{2n}(x)-\Wn(x)$ in \eqref{lambdaestimate} has already
been proved in Lemma \ref{sandwichlemma}. Thus, $\displaystyle\lim_{n\to+\infty}\lambda_{2n}(x)=\Wn(x)$. Similarly, due to \eqref{lambdarecursiondef} and \eqref{functional}, for the odd-indexed subsequence and $n\to+\infty$ we have
\[
\lambda_{2n+1}(x)= \ln(x)-\ln(\lambda_{2n}(x))\to \ln(x)-\ln(\Wn(x))=\Wn(x).
\]

\subsection{The proof of Lemma \ref{lem41}}\label{section411stproof}

Lemma \ref{lemmaWsimple} shows that $0<\be_0(x)<\Wn(x)$.\\
\noindent \textbf{Step 1.} Suppose that $0<\be_n(x)<\Wn(x)$ for some $n\in\Nn$. Then, due to \eqref{monotoneequivalence}, we have
$\be_n(x) e^{\be_n(x)}<\Wn(x) e^{\Wn(x)}=x$, so $1+\be_n(x)<1+\ln\left(\frac{x}{\be_n(x)}\right)$, implying 
\[
\be_n(x)<\frac{\be_n(x)}{1+\be_n(x)}\left(1+\ln\left(\frac{x}{\be_n(x)}\right)\right)=\be_{n+1}(x).
\]
\noindent \textbf{Step 2.} Suppose that $0<\be_n(x)<\Wn(x)$ for some $n\in\Nn$. Then---by using \eqref{functional} in the brackets $[\ldots]$ below---we have
\[
\Wn(x)-\be_{n+1}(x)=\Wn(x)-\frac{\be_n(x)}{1+\be_n(x)}\left(1+\ln\left(\frac{x}{\be_n(x)}\right)\right)=
\]
\[
\frac{1}{1+\be_n(x)}\left(\be_n(x)\left[\Wn(x)-\ln\left(\frac{x}{\be_n(x)}\right)\right]+\Wn(x)-\be_n(x)\right)=
\]
\[
\frac{1}{1+\be_n(x)}\left(\be_n(x)\left[\ln(x)-\ln(\Wn(x))-\ln\left(\frac{x}{\be_n(x)}\right)\right]+\Wn(x)-\be_n(x)\right)=
\]
\begin{equation}\label{W0-betan+1identity}
\frac{1}{1+\be_n(x)}\left(\be_n(x)\ln\left(\frac{\be_n(x)}{\Wn(x)}\right)+\Wn(x)-\be_n(x)\right)=
\end{equation}
\[
\frac{\Wn(x)}{1+\be_n(x)}\left(y \ln\left(y\right)+1-y\right)
\]
with $y:={\be_n(x)}/{\Wn(x)}\in(0,1)$. But, due to Lemma \ref{0lemma}, $y\ln(y)+1-y>0$, so $\be_{n+1}(x)<\Wn(x)$.\\
\noindent \textbf{Step 3.} The recursive application of Steps 1--2 completes the proof.

\subsection{The proof of Lemma \ref{sect41lem42}}\label{section412ndproof}

First we use Lemma \ref{lem41} and the identity \eqref{W0-betan+1identity}, then set $z:=\frac{\Wn(x)-\be_n(x)}{\Wn(x)}\in(0,1)$ and use the elementary estimate $\ln(1-z)<-z$ for $z\in(0,1)$  to get
\[
0<\Wn(x)-\be_{n+1}(x)=\frac{\be_n(x)}{1+\be_n(x)}\ln\left(1-\frac{\Wn(x)-\be_n(x)}{\Wn(x)}\right)+\frac{\Wn(x)-\be_n(x)}{1+\be_n(x)}<
\]
\[
\frac{\be_n(x)}{1+\be_n(x)}\cdot\frac{-(\Wn(x)-\be_n(x))}{\Wn(x)}+\frac{\Wn(x)-\be_n(x)}{1+\be_n(x)}=\frac{(\Wn(x)-\be_n(x))^2}{(1+\be_n(x))\Wn(x)}.
\]

\subsection{The proof of Lemma \ref{sect41lem43}}\label{section413rdproof}

The estimate \eqref{lem431stest} simply follows from the definition of $\beta_0(x)$ in \eqref{x>estart} and the earlier estimate \eqref{olderbound} (given in \cite{hoofar}). 

As for \eqref{lem432ndest}, one could find the global maximum of $\frac{e}{e-1}\frac{\ln(\ln (x))}{\ln(x)}$ for $x>e$.  It can be easily shown via differentiation that for $x>e$ we have
\[
\frac{e}{e-1}\frac{\ln(\ln (x))}{\ln(x)}\le \frac{1}{e-1}\approx 0.582,
\]
with equality exactly for $x=e^e$. However, we maximize the quantity $\Wn(x)-\be_0(x)$ directly to get the sharper upper bound $\kappa_1\approx 0.3133$. By the formula for the derivative of the inverse function we have $\Wn'(x)=\frac{\Wn(x)}{x(\Wn(x)+1)}$, so 
\[
(\Wn-\be_0)'(x)=-\frac{\ln (x)-\Wn(x)-1}{x \ln (x)\cdot (\Wn(x)+1) }.
\]
Here the denominator is positive because $x>e$. As for the numerator, its derivative is
\[
(\ln-\Wn-1)'(x)=\frac{1}{x (\Wn(x)+1)}>0,
\]
and
\[
\ln (e)-\Wn(e)-1=-1,\quad \ln \left(e^{e+1}\right)-\Wn\left(e^{e+1}\right)-1=e-\Wn\left(e\cdot e^{e}\right)=0.
\]
This means that $\ln (x)-\Wn(x)-1$ is negative for $x\in\left(e,e^{e+1}\right)$, zero at $x=e^{e+1}$, and positive for $x\in\left(e^{e+1},+\infty\right)$. That is, the function $\Wn-\be_0$ is strictly increasing on $\left(e,e^{e+1}\right)$ and decreasing on $\left(e^{e+1},+\infty\right)$, hence it has a global maximum at $x=e^{e+1}$, and $\Wn\left(e^{e+1}\right)-\be_0\left(e^{e+1}\right)=\ln\left(1+1/e\right)=\kappa_1$. The proof is complete.

\subsection{The proof of Lemma \ref{sect42lem47}}\label{sect42lem47proof}

We need to prove that $0<\Wn(x)<1/5+x/e$ holds for any $0<x<e$. Due to \eqref{monotoneequivalence}, this is equivalent to 
\[
x=\Wn(x) e^{\Wn(x)}<\left(\frac{1}{5}+\frac{x}{e}\right)e^{1/5+x/e}.
\]
Let us set 
\[
f_{\ref{sect42lem47proof}}(x):=\left(\frac{1}{5}+\frac{x}{e}\right)e^{1/5+x/e}-x,
\]
and notice that $f_{\ref{sect42lem47proof}}''(x)>0$, $f_{\ref{sect42lem47proof}}(0)>0$, $f_{\ref{sect42lem47proof}}(e)>0$, so the strictly convex function $f_{\ref{sect42lem47proof}}$ is positive at both endpoints of the interval. By solving $f_{\ref{sect42lem47proof}}'(x)=0$ symbolically, we find that this equation has a unique root at $x^*=e \left(\Wn\left(e^2\right)-{6}/{5}\right)\in(0,e)$, corresponding to the global minimum of $f_{\ref{sect42lem47proof}}$ on $(0,e)$. After some simplification, we get that
\[
f_{\ref{sect42lem47proof}}(x^*)=-\frac{e}{\Wn\left(e^2\right)}-e \Wn\left(e^2\right)+\frac{11 e}{5},
\]
and verify (for example, by using the recursion of Section \ref{IBconvx>e}) that the right-hand side above is positive ($> 0.0017$). This means that $f_{\ref{sect42lem47proof}}>0$ on $(0,e)$, completing the proof.

\subsection{The proof of Lemma \ref{sect43w0<b0}}\label{sect43w0<b0lemmaproof}

The leftmost and rightmost inequalities are obvious.\\
\noindent \textbf{Step 1.} We prove the second inequality first. Since now $-1<\Wn(x)$ is also true, we have, due to \eqref{monotoneequivalence}, that $-1+\sqrt{1+e x}<\Wn(x)$ is equivalent to
\[
(-1+\sqrt{1+e x})\,e^{-1+\sqrt{1+e x}}< \Wn(x)e^{\Wn(x)}=x.
\]
After introducing the new variable $z:=\sqrt{1+e x}\in(0,1)$, the above inequality becomes the obvious one
\[
(z-1)e^{z-1}<\frac{z^2-1}{e}.
\]
\noindent \textbf{Step 2.} We now prove $\Wn(x)<\be_0(x)$ by using the following bijective reparametrization: for any $-1/e<x<0$ there is a unique $y\in(-1,0)$ such that $y e^y=x$, namely, $y=\Wn(x)$. So the inequality $\Wn(x)<\be_0(x)$ becomes
\[
y<\frac{y e^{y+1} \ln \left(1+\sqrt{1+y e^{y+1}}\right)}{1+y e^{y+1}+\sqrt{1+y e^{y+1}}}.
\]
The denominator of this fraction is positive, but $y<0$, so the above is equivalent to
\[
{1+y e^{y+1} +\sqrt{1+y e^{y+1}}-e^{y+1} \ln \left(1+\sqrt{1+y e^{y+1}}\right)}>0.
\]
This left-hand side vanishes at $y=-1$, so it is enough to prove that its derivative (no longer containing a logarithm) is positive for any $-1<y<0$, that is
\[
\frac{\sqrt{1+y e^{y+1}} \left(3-y-4 e^{-y-1}\right)-4 e^{-y-1}-e^{y+1}+y e^{y+1}-3 y+3}{2 \sqrt{1+ye^{y+1} } \left(1+\sqrt{1+ye^{y+1}}\right)}>0.
\]
Here again, the denominator is positive, and, unexpectedly, the numerator can be factorized to yield
 \[
 e^{-y-1} \left(e^{y+1}-1-\sqrt{1+y e^{y+1}}\right) \left(y e^{y+1}-e^{y+1}+2 \sqrt{1+y e^{y+1}}+2\right).
 \]
After some elementary manipulations, we see that each of the three factors above are positive for any $-1<y<0$, completing the proof.

\subsection{The proof of Lemma \ref{lem412}}\label{lem412lemmaproof}

We prove the lemma by induction. 
Lemma \ref{sect43w0<b0} shows that $-1<\Wn(x)<\be_0(x)<0$.\\
\noindent \textbf{Step 1.} Suppose that $-1<\Wn(x)<\be_n(x)<0$ for some $n\in\Nn$. Then \eqref{monotoneequivalence} implies 
$\be_n(x) e^{\be_n(x)}>\Wn(x) e^{\Wn(x)}=x$, so---by carefully noting that now $x, \be_n(x), \Wn(x)\in(-1,0)$---we get $1+\be_n(x)<1+\ln\left(\frac{x}{\be_n(x)}\right)$, therefore 
\[
\be_n(x)>\frac{\be_n(x)}{1+\be_n(x)}\left(1+\ln\left(\frac{x}{\be_n(x)}\right)\right)=\be_{n+1}(x).
\]
\noindent \textbf{Step 2.} Suppose that $-1<\Wn(x)<\be_n(x)<0$ for some $n\in\Nn$. Then
\[
\be_{n+1}(x)-\Wn(x)=\frac{1}{1+\be_n(x)}\left(\be_n(x)\left[\ln\left(\frac{x}{\be_n(x)}\right)-\Wn(x)\right]+\be_n(x)-\Wn(x)\right),
\]
and now \eqref{functional2} is used in the brackets $[\ldots]$ above to get
\begin{equation}\label{lemma412proofstep2id}
\frac{1}{1+\be_n(x)}\left(\be_n(x)\ln\left(\frac{\Wn(x)}{\be_n(x)}\right)+\be_n(x)-\Wn(x)\right)=
\end{equation}
\[
\frac{-\Wn(x)}{1+\be_n(x)}\left(y \ln\left(y\right)-y+1\right)
\]
with $y:={\be_n(x)}/{\Wn(x)}\in(0,1)$. Finally, due to Lemma \ref{0lemma}, $y\ln(y)+1-y>0$, so 
$\Wn(x)<\be_{n+1}(x)$.

\subsection{The proof of Lemma \ref{sect43lem413}}\label{lemma413proof}

The proof is analogous to that of Lemma \ref{sect41lem42}. This time we use Lemma \ref{lem412}, the identity \eqref{lemma412proofstep2id} with $z:=\frac{\be_n(x)-\Wn(x)}{\Wn(x)}\in(-1,0)$, and the  estimate $\ln(1+z)<z$ for $z\in(-1,0)$  to get 
\[
0<\be_{n+1}(x)-\Wn(x)=\frac{-\be_n(x)}{1+\be_n(x)}\ln(1+z)+\frac{\be_n(x)-\Wn(x)}{1+\be_n(x)}<
\]
\[
\frac{-\be_n(x)}{1+\be_n(x)}\cdot\frac{\be_n(x)-\Wn(x)}{\Wn(x)}+\frac{\be_n(x)-\Wn(x)}{1+\be_n(x)}=\frac{(\be_n(x)-\Wn(x))^2}{-\Wn(x)(1+\be_n(x))}.
\]

\subsection{The proof of Lemma \ref{sect43b0w0globmax}}\label{sect43b0w0globmaxlemmaproof}

The first two inequalities below follow from Lemma \ref{sect43w0<b0}:
\[
0<\be_0(x)-\Wn(x)<\be_0(x)-(-1+\sqrt{1+e x})=
\]
\begin{equation}\label{lem412proofaux}
\frac{e x\ln(1+\sqrt{1+e x})}{\sqrt{1+e x}(1+\sqrt{1+e x})}+1-\sqrt{1+e x}.
\end{equation}
It is thus sufficient to upper estimate \eqref{lem412proofaux}. Since $\ln(1+z)> z -\frac{z^2}{2}+\frac{z^3}{3}-\frac{z^4}{4}>0$ for $z\in(0,1)$, and we have $x<0$, by denoting $z:=\sqrt{1+e x}\in(0,1)$ we see that \eqref{lem412proofaux} is further increased by 
\begin{equation}\label{log4Taylor}
\frac{e x\left(z -\frac{z^2}{2}+\frac{z^3}{3}-\frac{z^4}{4}\right)}{z(1+z)}+1-z.
\end{equation}
But now $x=(z^2-1)/e$, so \eqref{log4Taylor} can be rewritten as
\[
-\frac{z(z-1)}{12} \left(3 z^2-4 z+6\right),
\]
and one checks that the global maximum of this quartic polynomial for $z\in(0,1)$ is less than $1/10$ (in fact, it is approximately $0.09928$).

\subsection{The proof of Lemma \ref{highpowerglobmax}}\label{highpowerglobmaxlemmaproof}

Due to $-1<\Wn(x)<\be_0(x)<0$, we have
\[
0<\frac{\be_0(x)-\Wn(x)}{-\Wn(x)\sqrt{1+e x}}<\frac{\be_0(x)-\Wn(x)}{-\be_0(x)\sqrt{1+e x}},
\]
so it is enough to prove that the rightmost expression above is less than $1/10$. This sufficient condition can be rearranged into the form
\begin{equation}\label{lemma416proofaux1}
\frac{e x \left(10+\sqrt{1+e x}\right) \ln \left(1+\sqrt{1+e x}\right)}{10 \left(1+e x+\sqrt{1+e x}\right)}<\Wn(x).
\end{equation}
Let us use again the parametrization $y e^y=x$ with $y\in(-1,0)$ as in Step 2 of the proof of Lemma \ref{sect43w0<b0}. Then \eqref{lemma416proofaux1} becomes
\[
\frac{y e^{y+1} \left(10+\sqrt{1+y e^{y+1}}\right) \ln \left(1+\sqrt{1+y e^{y+1}}\right)}{10 \left(1+y e^{y+1}+\sqrt{1+y e^{y+1}}\right)}<y,
\] 
or, since the denominator is positive for $0<y<1$,
\begin{equation}\label{lemma416proofaux2}
\ln \left(1+\sqrt{1+y e^{y+1}}\right)-\frac{10 \left(1+y e^{y+1}+\sqrt{1+y e^{y+1}}\right)}{e^{y+1} \left(10+\sqrt{1+y e^{y+1}}\right)}>0.
\end{equation}
The left-hand side of \eqref{lemma416proofaux2} vanishes at $y=-1$, so it is enough to prove that its derivative is positive for $-1<y<0$. This derivative can be written as
\[
\frac{e^{-y-1}}{2 \sqrt{1+y e^{y+1}} \left(1+\sqrt{1+y e^{y+1}}\right) \left(10+\sqrt{1+y e^{y+1}}\right)^2}\cdot f_{\ref{highpowerglobmaxlemmaproof}}(y,z),
\]
where
\[
f_{\ref{highpowerglobmaxlemmaproof}}(y,z):=440+10 \sqrt{1+y z} \left(\left(y^2+y+2\right) z^2+(15 y-31) z+44\right)+
\]
\[
y^2 z^2 (z+30)+y z\left(z^2-109 z+370\right)+101 z^2-310z
\]
with $z:=e^{y+1}\in(1,e)$. Then we also have $-1<yz<0$. Clearly, to finish the proof, it suffices to prove that  
$f_{\ref{highpowerglobmaxlemmaproof}}(y,z)>0$ for any $-1<y<0$, $1<z<e$ and $-1<yz<0$. Now, by introducing the new variable $w:=\sqrt{1+y z}\in(0,1)$, the expression $f_{\ref{highpowerglobmaxlemmaproof}}(y,z)$ becomes
\[
 (w+10 )^2 z^2 + (w+1) (w^3+ 9 w^2- 120 w-200  ) z +10 (w+1 )^3 ( w^2+10),
\]
so it is enough to prove that this bivariate polynomial is positive for any $1<z<e$ and $0<w<1$. But its discriminant with respect to $z$, 
$w^2 (w+1)^2 (  w^4- 22 w^3- 999 w^2- 7760 w -1600 )$, is trivially negative for $w\in(0,1)$, completing the proof.

\subsection{The proof of Lemma \ref{sect44lem421}}\label{sect44lem421lemmaproof}

\noindent \textbf{Step 1.} First we prove
\begin{equation}\label{lemma420proof1}
\Wm(x)< \ln (-x)-\ln (-\ln (-x))
\end{equation}
for $x\in(-1/e,0)$ (instead of only for $x\in(-1/4,0)$). Although \eqref{lemma420proof1} is identical to \eqref{W-1bound}, we provide a direct proof for the sake of completeness. By using the bijective reparametrization $x=y e^y$ mentioned in the beginning of Section \ref{subsection44}, \eqref{lemma420proof1} is equivalent to 
\[
y<\ln \left(-y e^y\right)-\ln \left(-\ln \left(-y e^y\right)\right)\quad\quad \text{for \ \  } y<-1,
\]
that is, to
$0<\ln \left(\frac{-y}{-\ln \left(-y e^y \right)}\right)$, which reduces to the obvious inequality $-y>-\ln \left(-y e^y\right)$.\\
\noindent \textbf{Step 2.} Now we prove that
\[
f_{\ref{sect44lem421lemmaproof}}(x):=\ln (-x)-\ln (-\ln (-x))-\Wm(x)<1/2,
\]
again, for any $x\in(-1/e,0)$. We have
\[
f'_{\ref{sect44lem421lemmaproof}}(x)=\frac{f_{\ref{sect44lem421lemmaproof}1}(x)}{x \ln (-x)\left(\Wm(x)+1\right)}
\]
with
\[
f_{\ref{sect44lem421lemmaproof}1}(x):=-\Wm(x)+\ln (-x)-1.
\]
The denominator of $f'_{\ref{sect44lem421lemmaproof}}$ is clearly negative. Moreover,
\[
f'_{\ref{sect44lem421lemmaproof}1}(x)=\frac{1}{x \left(\Wm(x)+1\right)}>0,
\]
so $f_{\ref{sect44lem421lemmaproof}1}$ is strictly increasing on $(-1/e,0)$. But we notice that 
\[f_{\ref{sect44lem421lemmaproof}1}(-e^{1-e})=-\Wm(-e^{1-e})-e=0,\] so $f_{\ref{sect44lem421lemmaproof}1}<0$ on $(-1/e,-e^{1-e})$ and $f_{\ref{sect44lem421lemmaproof}1}>0$ on $(-e^{1-e},0)$. 

This means that $f_{\ref{sect44lem421lemmaproof}}$ is strictly increasing on $(-1/e,-e^{1-e})$, strictly decreasing on $(-e^{1-e},0)$, and it has a global maximum at $x=-e^{1-e}\approx -0.179$ (we remark that 
$f_{\ref{sect44lem421lemmaproof}}(-1/e)$ $=$ $\displaystyle\lim_{0^-}f_{\ref{sect44lem421lemmaproof}}=0$). Since $f_{\ref{sect44lem421lemmaproof}}(-e^{1-e})=1-\ln (e-1)<1/2$, the proof of Step 2 is complete.\\
\noindent \textbf{Step 3.} Next, we show that
\[
f_{\ref{sect44lem421lemmaproof}2}(x):=-1-\sqrt{2} \sqrt{1+e x}-\Wm(x)>0
\]
holds for any $x\in(-1/e,0)$. To this end, we first verify that $f_{\ref{sect44lem421lemmaproof}2}$ is strictly increasing on $(-1/e,0)$.

After applying the reparametrization  $x=y e^y$, and noticing that $y\mapsto y e^y$ is strictly decreasing on $(-\infty,-1)$, we need to verify that  
\[
f_{\ref{sect44lem421lemmaproof}3}(y):=-1-\sqrt{2} \sqrt{1+y e^{y+1}}-y
\]
is strictly decreasing on $(-\infty,-1)$. But 
\[
f'_{\ref{sect44lem421lemmaproof}3}(y)=-1+f_{\ref{sect44lem421lemmaproof}4}(y),
\]
with
\[
f_{\ref{sect44lem421lemmaproof}4}(y):=-\frac{e^{y+1} (y+1)}{\sqrt{2 ye^{y+1} +2}},
\]
so it is enough to show that $f_{\ref{sect44lem421lemmaproof}4}(y)<1$ for any $y<-1$. Clearly, $f_{\ref{sect44lem421lemmaproof}4}(y)<1$ is equivalent to $-e^{y+1} (y+1)<\sqrt{2 y e^{y+1} +2}$, and here both sides are positive---so squaring the inequality is allowed, reducing it to  $0<-e^{2 y+2} (y+1)^2+2 y e^{y+1}+2$. After the substitution $z:=y+1<0$, we are to show $-e^{2 z} z^2+2 e^z (z-1)+2>0$. The left-hand side here vanishes at $z=0$, and its derivative is $-2 e^z z \left(e^z( z+1)-1\right)<0$, finishing the claim. 

Now, as the strict monotonicity of $f_{\ref{sect44lem421lemmaproof}2}$ has been established, notice that $f_{\ref{sect44lem421lemmaproof}2}(-1/e)=0$, so Step 3 is complete.\\
\noindent \textbf{Step 4.} Finally, we show that
\[
-1-\sqrt{2} \sqrt{1+e x}-\Wm(x)<1/2
\]
for any $-1/e<x\le-1/4$. In Step 3 we proved that the left-hand side, $f_{\ref{sect44lem421lemmaproof}2}$ is strictly increasing on $(-1/e,0)$, so it is sufficient to show that 
$f_{\ref{sect44lem421lemmaproof}2}(-1/4)<1/2$. But this last inequality is equivalent to $\Wm\left(-1/4\right)>\left(-3-\sqrt{8-2 e}\right)/2$, being true due to \eqref{reversemonotoneequivalence}, 
hence completing the proof of the lemma.

\subsection{The proof of Lemma \ref{sect44lem419}}\label{sect44lem419lemmaproof}

In \eqref{sect44lem419mon}, the inequality $\be_0(x)<-1$ is elementary, and $\Wm(x)<\be_0(x)$ has been proved in Lemma \ref{sect44lem421}, so we have
\begin{equation}\label{lemma421proofstep0initialinequality}
\Wm(x)<\be_0(x)<-1.
\end{equation}
Now let us formulate two conditional statements in Steps 1a and 1b, to be used in Step 2.

\noindent \textbf{Step 1a.} We claim that if 
\begin{equation}\label{lemma421proofstep1assumption}
\be_n(x)<\Wm(x)<-1
\end{equation}
for some $n\in\mathbb{N}^+$, then $\be_n(x)<\be_{n+1}(x)$. 

Indeed, due to \eqref{reversemonotoneequivalence}, the assumption \eqref{lemma421proofstep1assumption} implies 
$\be_n(x) e^{\be_n(x)}>\Wm(x) e^{\Wm(x)}=x$. By taking into account $x<0$, $\be_n(x)<0$, and $1+\be_n(x)<0$, this leads to  $1+\be_n(x)<1+\ln\left(\frac{x}{\be_n(x)}\right)$, that is, to 
\[
\be_n(x)<\frac{\be_n(x)}{1+\be_n(x)}\left(1+\ln\left(\frac{x}{\be_n(x)}\right)\right)=\be_{n+1}(x).
\]
\noindent \textbf{Step 1b.} Assume in this step that we have $\be_n(x)<-1$ for some $n\in\mathbb{N}$. 

Then $\be_{n+1}(x)$ is well-defined, real, and clearly satisfies
\[
\Wm(x)-\be_{n+1}(x)=\frac{1}{1+\be_n(x)}\left(\Wm(x)-\be_n(x)+\be_n(x)\left[\Wm(x)-\ln\left(\frac{x}{\be_n(x)}\right)\right]\right).
\]
Now by using \eqref{functional3}, $x<0$, $\be_n(x)<0$, and $\Wm(x)<0$, the expression in $[\ldots]$ above is $\ln\left(\frac{\be_n(x)}{\Wm(x)}\right)$, hence 
\begin{equation}\label{lemma421proofstep2aidentity1}
\Wm(x)-\be_{n+1}(x)=\frac{1}{1+\be_n(x)}\left(\Wm(x)-\be_n(x)+\be_n(x)\ln\left(\frac{\be_n(x)}{\Wm(x)}\right)\right),
\end{equation}
or, in other words,
\begin{equation}\label{lemma421proofstep2aidentity2}
\Wm(x)-\be_{n+1}(x)=\frac{\Wm(x)}{1+\be_n(x)}\left(1-\frac{\be_n(x)}{\Wm(x)}+\frac{\be_n(x)}{\Wm(x)}\ln\left(\frac{\be_n(x)}{\Wm(x)}\right)\right).
\end{equation}
\noindent \textbf{Step 2a.} Since $\be_0(x)<-1$ due to \eqref{lemma421proofstep0initialinequality}, we can  consider \eqref{lemma421proofstep2aidentity2} with $n=0$ and with $y:=\frac{\be_0(x)}{\Wm(x)}$. Then $y\in(0,1)$ and $\frac{\Wm(x)}{1+\be_0(x)}>0$, due to \eqref{lemma421proofstep0initialinequality} again. From these, by using Lemma \ref{0lemma}, we conclude that $-1>\Wm(x)>\be_{1}(x)$.\\
\noindent \textbf{Step 2b.} Assume \eqref{lemma421proofstep1assumption}, as an inductive hypothesis,  for some $n\in\mathbb{N}^+$.
For $n=1$, this has been proved in Step 2a, so the induction can be started. Then Step 1a shows that $\be_{n+1}(x)$ is well-defined, real, and satisfies $\be_n(x)<\be_{n+1}(x)$. Moreover---since the assumption of Step 1b is fulfilled---we can apply \eqref{lemma421proofstep2aidentity2} with $y:=\frac{\be_n(x)}{\Wm(x)}$. Then $y>1$ and $\frac{\Wm(x)}{1+\be_n(x)}>0$ are both consequences of the inductive hypothesis  \eqref{lemma421proofstep1assumption}, so Lemma \ref{0lemma} implies $\be_{n+1}(x)<\Wm(x)<-1$.

By taking into account \eqref{lemma421proofstep0initialinequality} also, the above induction verifies \eqref{sect44lem419mon} and the left inequality of \eqref{sect44lem419est} for any $n\in\mathbb{N}^+$.\\
\noindent \textbf{Step 2c.} Let us show the second inequality in \eqref{sect44lem419est} for $n=1$. 

Notice that the assumption of Step 1b is fulfilled because of \eqref{lemma421proofstep0initialinequality}, so we apply \eqref{lemma421proofstep2aidentity1} with $n=0$ and get
\begin{equation}\label{lemma421proofstep2cformula}
\Wm(x)-\be_{1}(x)=\frac{1}{1+\be_0(x)}\left(\Wm(x)-\be_0(x)+\be_0(x)\ln\left(1-z\right)\right)
\end{equation}
with $z:=\frac{\Wm(x)-\be_0(x)}{\Wm(x)}$. Due to \eqref{lemma421proofstep0initialinequality} again, we have $z\in(0,1)$, so we can use the elementary inequality $\ln(1-z)<-z$ (and $1+\be_0(x)<0$) to estimate \eqref{lemma421proofstep2cformula} as
\[
\Wm(x)-\be_{1}(x)<\frac{\Wm(x)-\be_0(x)}{1+\be_0(x)}-\frac{\be_0(x)}{1+\be_0(x)}\cdot\frac{\Wm(x)-\be_0(x)}{\Wm(x)}=
\]
\[
\frac{(\be_0(x)-\Wm(x))^2}{(1+\be_0(x))\Wm(x)}=\left(\be_0(x)-\Wm(x)\right)\cdot \left(\frac{\be_0(x)-\Wm(x)}{|\Wm(x)|\cdot |1+\be_0(x)|}\right),
\]
completing Step 2c.\\
\noindent \textbf{Step 2d.} Finally, we prove the second inequality in \eqref{sect44lem419est} for any $n\ge 2$ by induction. 

The induction can be started, since the second inequality in \eqref{sect44lem419est} for $n=1$ is Step 2c. So let us suppose that 
\begin{equation}\label{lemma421proofstep2finitial}
\Wm(x)-\be_{n}(x)<\left(\be_0(x)-\Wm(x)\right)\cdot \left(\frac{\be_0(x)-\Wm(x)}{|\Wm(x)|\cdot |1+\be_0(x)|}\right)^{-1+2^n}
\end{equation}
holds for some $n\ge 1$. Then we can apply \eqref{lemma421proofstep2aidentity1} (since the assumption of Step 1b is satisfied due to \eqref{sect44lem419mon} we already know), hence
\[
\Wm(x)-\be_{n+1}(x)=\frac{1}{1+\be_n(x)}\left(\Wm(x)-\be_n(x)+\be_n(x)\ln\left(1-z\right)\right),
\]
with $z:=\frac{\Wm(x)-\be_n(x)}{\Wm(x)}$. This time, however, we have $z<0$ due to \eqref{sect44lem419mon}. Nevertheless, the inequality $\ln(1-z)<-z$ still holds, so 
\begin{equation}\label{lemma421proofstep2fn+1}
\Wm(x)-\be_{n+1}(x)<\frac{\Wm(x)-\be_n(x)}{1+\be_n(x)}-\frac{\be_n(x)}{1+\be_n(x)}\cdot \frac{\Wm(x)-\be_n(x)}{\Wm(x)},
\end{equation}
where we have also taken into account that $\frac{\be_n(x)}{1+\be_n(x)}>0$ (being a consequence \eqref{sect44lem419mon}). But the right-hand side of \eqref{lemma421proofstep2fn+1} is equal to $\frac{(\Wm(x)-\be_n(x))^2}{|\Wm(x)|\cdot |1+\be_n(x)|}$, so we proved
\[
\Wm(x)-\be_{n+1}(x)<\frac{(\Wm(x)-\be_n(x))^2}{|\Wm(x)|\cdot |1+\be_n(x)|}.
\]
Notice now that---due to \eqref{sect44lem419mon}---we have $\frac{1}{|1+\be_n(x)|}<\frac{1}{|1+\be_0(x)|}$, therefore 
\begin{equation}\label{lemma421proofstep2ffinal}
\Wm(x)-\be_{n+1}(x)<\frac{(\Wm(x)-\be_n(x))^2}{|\Wm(x)|\cdot |1+\be_0(x)|}.
\end{equation}
The left-hand side of \eqref{lemma421proofstep2finitial} is positive (due to \eqref{sect44lem419mon}), so we can combine \eqref{lemma421proofstep2ffinal} and \eqref{lemma421proofstep2finitial} to get
\[
\Wm(x)-\be_{n+1}(x)<\frac{1}{|\Wm(x)|\cdot |1+\be_0(x)|}\cdot\frac{\Big(\be_0(x)-\Wm(x)\Big)^{2^{n+1}}}{\big(|\Wm(x)|\cdot |1+\be_0(x)|\big)^{-2+2^{n+1}}}=
\]
\[
\left(\be_0(x)-\Wm(x)\right)\cdot \left(\frac{\be_0(x)-\Wm(x)}{|\Wm(x)|\cdot |1+\be_0(x)|}\right)^{-1+2^{n+1}},
\]
completing the induction, and the proof of the lemma.

\subsection{The proof of Lemma \ref{sect44lem422}}\label{sect44lem422lemmaproof}

\noindent \textbf{Step 1.} Let us first consider the case $-1/4<x< 0$. Then, due to Lemma \ref{sect44lem421} and \eqref{sect44lem419mon}, we have
\[
0<\frac{\be_0(x)-\Wm(x)}{|\Wm(x)|\cdot |1+\be_0(x)|}<\frac{1/2}{|\Wm(x)|\cdot |1+\be_0(x)|}<\frac{1/2}{|\be_0(x)|\cdot |1+\be_0(x)|},
\]
proving \eqref{sect44lem22estsharper}. Moreover, it is elementary to check that both $x\mapsto |\ln (-x)-\ln (-\ln (-x)) |$ and $x\mapsto | 1+\ln (-x)-\ln (-\ln (-x))|$ are strictly increasing for $-1/4<x< 0$, and their product satisfies $|\be_0(-1/4)|\cdot |1+\be_0(-1/4)|>1$, so \eqref{sect44lem22estsharper} implies \eqref{sect44lem22est} for $-1/4<x< 0$.\\
\noindent \textbf{Step 2.} Let us consider now the case $-1/e<x\le-1/4$. Then \eqref{sect44lem419mon} yields
\[
0<\frac{\be_0(x)-\Wm(x)}{|\Wm(x)|\cdot |1+\be_0(x)|}<\frac{\be_0(x)-\Wm(x)}{|1+\be_0(x)|}=
\frac{ -1-\sqrt{2} \sqrt{1+e x}-\Wm(x)}{\sqrt{2} \sqrt{1+e x}},
\]
so to prove \eqref{sect44lem22est}, it is sufficient to show that the right-hand side above is less than $1/2$. This last sufficient condition (RHS $<1/2$) is equivalent to
\[
2 \Wm(x)+3 \sqrt{2+2 e x}+2>0,
\]
which, after the reparametrization $x=y e^y$, becomes
\begin{equation}\label{lemma422Step2proof}
2 y+3 \sqrt{2+2 y e^{y+1} }+2>0.
\end{equation}
It is enough to prove \eqref{lemma422Step2proof} for $-2.2<y<-1$, because the range of the function
$(-2.2,-1) \ni y\mapsto y e^y$ includes the interval $(-1/e,-1/4]$. But for $-2.2<y<-1$,  \eqref{lemma422Step2proof} is equivalent to $0<2+2 y e^{y+1}-\left(\frac{2 y+2}{3}\right)^2$, or, to
$9-2 z^2+9 e^z (z-1)>0$ after the shift $z:=y+1\in(-1.2,0)$. On this interval, the degree-6 Taylor polynomial of the exponential function about the origin is greater than $e^z$, so $9-2 z^2+9 e^z (z-1)$ is decreased by replacing $e^z$ with its degree-6 Taylor approximation. This way we get $\frac{1}{80} z^2 \left(z^2+5 z+10\right) \left(z^3+14 z+20\right)$, which is clearly positive for $z\in(-1.2,0)$, completing the proof.

\subsection{The proof of Lemma \ref{smallWestimate}}\label{section33}

The first, fourth and fifth inequalities are elementary, and the third one is analogous
to the second one, so here we prove only the second inequality.
Since on  $(-1/4,0)$ we have
$-1<\frac{1}{2} \left(\sqrt{1+4 x}-1\right)$ and $-1<\Wn(x)$,  due to 
\eqref{monotoneequivalence}, 
$\frac{1}{2} \left(\sqrt{1+4 x}-1\right)<\Wn(x)$ is equivalent to
\begin{equation}\label{second33}
\frac{1}{2} \left(\sqrt{1+4 x}-1\right)\exp\left({\frac{1}{2} \left(\sqrt{1+4 x}-1\right)}\right) <x.
\end{equation}
Now there is a unique $y\in (-1/2,0)$ such that $y=(\sqrt{1+4 x}-1)/2$, so \eqref{second33}
is equivalent to $y e^y<y(1+y)$, that is to $e^y>1+y$. The proof is complete.

\subsection{The proof of Lemma \ref{betalemma}}\label{section31}
 
For any $x>e$, $\beta(x)>-1$ is equivalent to
\[
f_{\ref{section31}}(x):=\frac{x(x-1)}{x+e^2-2 e}-\ln (x)>0.
\]
Notice that $f_{\ref{section31}}(e)=0$. The proof will be complete as soon as we show that $f_{\ref{section31}}'>0$ on $(e,+\infty)$.
We have  
$f_{\ref{section31}}'(x)=N_{\ref{section31}}(x)/\left(x \left(x+e^2-2e\right)^2\right)$
with 
\[
N_{\ref{section31}}(x):=x^3+(2 (e-2) e-1) x^2-3 (e-2) e x-(e-2)^2 e^2.
\]
But $N_{\ref{section31}}$ has 3 real roots on $(-\infty,2)$, so $N_{\ref{section31}}>0$ on $(e,+\infty)$, completing the proof. 

\subsection{The proof of Lemma \ref{lemma25}}\label{section32}

 Let us denote the left-hand side of \eqref{lemma25LHS} by $f_{\ref{section32}}(x)$. Notice that
$\displaystyle f_{\ref{section32}}(e)=0=\lim_{\infty} f_{\ref{section32}}$, so the proof is finished as soon as we have 
shown that there is 
a unique $x_{\ref{section32}}^*>e$ such that $f_{\ref{section32}}$ is strictly increasing on $(e,x_{\ref{section32}}^*)$ 
and strictly decreasing on $(x_{\ref{section32}}^*,+\infty)$  (we remark that $x_{\ref{section32}}^*\approx 22.04$).
To this end, we write the derivative of $f_{\ref{section32}}$ as $f_{\ref{section32}}'=N_{\ref{section32}}/D_{\ref{section32}}$ with
\[
N_{\ref{section32}}(x):=(x-1) P_{\ref{section32}1}(x)-P_{\ref{section32}2}(x) \ln (x),
\]
\[
D_{\ref{section32}}(x):=x^2 (x-1)^2 \left(x+e^2-2 e\right),
\]
\[
P_{\ref{section32}1}(x):=((e-2) e+2) x^2+2 (e-2) e x+(e-2)^2 e^2
\]
and
\[
P_{\ref{section32}2}(x):=(x+(e-2) e) \left(x^2+2 (e-2) e x+(2-e) e\right).
\]
We check that $D_{\ref{section32}}$ and $P_{\ref{section32}2}$ are both positive on $(e,+\infty)$, so the sign of 
$f_{\ref{section32}}'(x)$ is determined by that of $N_{\ref{section32}}(x)$, or by that of
\[
f_{\ref{section32}3}(x):=\frac{(x-1) P_{\ref{section32}1}(x)}{P_{\ref{section32}2}(x)}-\ln (x).
\] 
Now for $x>e$ we have that
\[
f_{\ref{section32}3}'(x)=-\frac{(x-e) (x+e-2) P_{\ref{section32}4}(x)}{x P_{\ref{section32}2}^2(x)}
\]
with
\[
P_{\ref{section32}4}(x):=x^4-e(e-2) (3 (e-2) e-1) x^3-e(e-2)  ((e-2) e-1) (1+4 (e-2) e) x^2+\]
\[
(e-2)^2 e^2 (1+5 (e-2) e) x-(e-2)^3 e^3.
\]
The polynomial $P_{\ref{section32}4}$ has 3 real roots in $(-\infty,e)$ and a unique root in 
$(e,+\infty)$ at $x_{\ref{section32}4}^*$ (we have $x_{\ref{section32}4}^*\approx 10.67$). From this we see that 
$-P_{\ref{section32}4}$ is positive on $(e,x_{\ref{section32}4}^*)$ and negative on $(x_{\ref{section32}4}^*,+\infty)$. This means that
$f_{\ref{section32}3}'$ is positive on $(e,x_{\ref{section32}4}^*)$ and negative on $(x_{\ref{section32}4}^*,+\infty)$. But 
$f_{\ref{section32}3}(e)=0$ and $\displaystyle \lim_{\infty} f_{\ref{section32}3}=-\infty$, so $f_{\ref{section32}3}$ has a unique root
at some $x_{\ref{section32}}^*$ with $x_{\ref{section32}}^*>x_{\ref{section32}4}^*>e$, and $f_{\ref{section32}3}>0$ on $(e,x_{\ref{section32}}^*)$ and $f_{\ref{section32}3}<0$ on $(x_{\ref{section32}}^*,+\infty)$. Therefore, $f_{\ref{section32}}'>0$ on $(e,x_{\ref{section32}}^*)$ and $f_{\ref{section32}}'<0$ on $(x_{\ref{section32}}^*,+\infty)$,
and the proof is complete.

\subsection{Auxiliary estimates}\label{section35}

Here we state some auxiliary inequalities used in the earlier sections. For brevity, some elementary proofs are omitted.

\begin{lem}\label{0lemma}
For $y\in(0,1)\cup(1,+\infty)$ one has
$
y \ln (y)+1-y>0
$.
\end{lem}
\begin{proof}
We set $g(y):=y \ln (y)+1-y$. Since $g'(y)=\ln(y)<0$ for $0<y<1$, $g'(y)>0$ for $y>1$, and $g(1)=0$, the proof is complete.
\end{proof}

\begin{lem}\label{1lemma}
Inequality $y\ge 1$ implies
$
y^2-y \ln (y)+\ln (y)>0
$.
\end{lem}
\begin{lem}\label{72lemmasimple}
For $1<y\le 7/2$ one has
$
y+(1-y) \ln(y)>0
$.
\end{lem}

\begin{lem}\label{72lemma}
For $y>7/2$ we have 
\begin{equation}\label{lemma72label}
2 y^3-2 y^2 \ln(y)+y(\ln^2(y)-\ln(y)+1)-\ln^2(y)+\ln(y)>0.
\end{equation}
\end{lem}
\begin{proof} We set
$P_{\ref{section35}}(y,z):=(y-1) z^2+\left(-2 y^2-y+1\right) z+(2 y^3+y)$, then the left-hand side of
\eqref{lemma72label} is equal to $P_{\ref{section35}}(y,\ln(y))$. We prove that $P_{\ref{section35}}(y,z)>0$
for $y>7/2$ and $z\in\mathbb{R}$. For any $y>7/2$, the quadratic polynomial 
$z\mapsto P_{\ref{section35}}(y,z)$
has positive leading coefficient, and its global minimum is located at
$z_{\ref{section35}}:=-\frac{-2 y^2-y+1}{2 (y-1)}$. But 
\[
P_{\ref{section35}}(y,z_{\ref{section35}})=\frac{20 y^4+4 y^3-5 y^2-10 y+3}{4 (y-1)}>0
\]
for $y>7/2$.
\end{proof}

\begin{lem}\label{56lemma}
For $y>0$ we have
$
\ln (y)<\frac{5 }{6} \sqrt{y}$.
\end{lem}
\begin{lem}\label{elementarylemma}
On the interval $(e,+\infty)$, the following inequalities hold:
\begin{equation}\label{ineq1}
L_1-L_2+\frac{L_2}{L_1}>L_1-L_2>1,
\end{equation}
\begin{equation}\label{ineq2}
-\frac{4}{5}<\frac{L_2}{L_1^2}-\frac{L_2}{L_1}<0,
\end{equation}
\begin{equation}\label{ineq3}
0<\frac{e}{e-1}\frac{L_2}{L_1}<1.
\end{equation}
\end{lem}
\begin{proof} We have $(L_1-L_2)'(x)=\frac{\ln (x)-1}{x \ln (x)}>0$ and 
$(L_1-L_2)(1)=1$, proving \eqref{ineq1}. 

The upper bound in \eqref{ineq2} is just
$\left(\frac{L_2}{L_1^2}-\frac{L_2}{L_1}\right)(x)=-\frac{(\ln (x)-1) \ln(\ln (x))}{\ln^2(x)}<0$, while the lower
bound in \eqref{ineq2} is equivalent to the elementary inequality $0<4 y^2-5 (y-1) \ln(y)$
with $y:=\ln(x)$ for $y>1$. 

The lower bound in \eqref{ineq3} is trivial. With $y:=\ln(x)$ again, the upper bound is the elementary inequality $\ln (y)<\frac{(e-1) y}{e}$.
\end{proof}

\begin{lem}\label{2ln2lemma}
For any $z\in\left[ 0, \ln(2)\right]$, one has
$-z\le \ln\left(1-\frac{z}{2\ln(2)}\right)$.
\end{lem}
\begin{proof} Notice that $-z=\ln\left(1-\frac{z}{2\ln(2)}\right)$ holds for $z=0$ and $z=\ln(2)$. 
Moreover, the second derivative of the right-hand side is negative on 
$\left[ 0, \ln(2)\right]$, so $z\mapsto\ln\left(1-\frac{z}{2\ln(2)}\right)$ is concave. The proof is complete.
\end{proof}
\begin{lem}\label{Winductiveestimatelemma}
For any $x\ge e$ and $m\in\mathbb{N}^+$, we have
\[
0\le \frac{\ln(x)-\Wn(x)}{\Wn^{\,m}(x)}\le \frac{1}{\sqrt{2m}}.
\]
\end{lem}
\begin{proof} We know from Lemma \ref{lemmaWsimple} 
that
$\frac{L_1-\Wn}{\Wn^m}$ is non-negative on $[e,+\infty)$. As for the upper estimate, 
let us consider the chain
\[
\Wn^m\ge (L_1-L_2)^m \ \boxed{\ge} \ \sqrt{2m} \, L_2 \ge \sqrt{2m}\, (L_1-\Wn).
\]
Here the first and third inequalities hold due to  
Lemma \ref{lemmaWsimple} again (taking also into account that $L_1-L_2>0$ because of \eqref{ineq1}), so it is enough to show  the second one.
But inequality $ \boxed{\ge}$ is equivalent to
$\left(e^z-z\right)^m\ge z\sqrt{2m}$,
to be proved for any $z:=\ln(\ln(x))\ge 0$. By using the series expansion of $\exp$ around $0$ and the binomial theorem we get
\[
\left(e^z-z\right)^m\ge \left(1+\frac{z^2}{2}\right)^m\ge 1+m\cdot \frac{z^2}{2}=
\left(1-z\sqrt{\frac{m}{2}}\right)^2+z\sqrt{2m}\ge z \sqrt{2m},
\]
completing the proof.
\end{proof}

\begin{rem} One can actually prove an estimate which is sharper than the one in Lemma \ref{Winductiveestimatelemma}. Namely, for any $x\ge e$ and $m\in\mathbb{N}^+$ we have
\[
0\le \frac{\ln(x)-\Wn(x)}{\Wn^{\,m}(x)}\le \frac{1}{m e},
\]
and equality in the upper estimate occurs exactly for $\displaystyle x=\sqrt[m]{e}\cdot \exp\left(\sqrt[m]{e}\right)$.
\end{rem}

\begin{lem}\label{thm23auxupperest}
For any $x\ge x^*\approx 6288.69$  (defined in Section \ref{refinedapproximationsW}), we have
\[
L_1(x)-\ln
   \left(L_1(x)-L_2(x)+\frac{L_2(x)}{L_1(x)}\right)<L_1(x)-L_2(x)+\frac{L_2(x)}{L_1(x)}+\frac{\left(L_2(x)-2\right) L_2(x)}{2 L_1(x){}^2}+ \frac{L_2(x){}^3}{L_1(x){}^3}.
\]
\end{lem}
\begin{proof} By taking the difference of the two sides LHS$-$RHS above, and introducing the new variable $y:=\ln(x)$, it is enough to prove that
\[
f_{\ref{section35}1}(y):=-\ln \left(y-\ln (y)+\frac{\ln (y)}{y}\right) +\ln(y)  -\frac{\ln (y)}{y}-\frac{(\ln (y)-2) \ln (y)}{2 y^2}-\frac{\ln ^3(y)}{y^3}<0
\]
for, say, $y\in(8,+\infty)$ (since $\ln(x^*)>8$). But $f_{\ref{section35}1}(8)<0$ and $\displaystyle \lim_{+\infty} f_{\ref{section35}1}=0$, so the proof will be finished as soon as we have shown that 
$f_{\ref{section35}1}'>0$ on $(8,+\infty)$, where
\begin{equation}\label{f'A91}
f_{\ref{section35}1}'(y)=\frac{f_{\ref{section35}2}(y) \cdot \ln (y)}{y^4 \left(y^2-y \ln (y)+\ln (y)\right)}
\end{equation}
with
\[
f_{\ref{section35}2}(y):=y-2 y^2+(y^2-3y) \ln(y)+\left(2 y^2+4 y-3\right) \ln^2(y)-3 (y-1) \ln^3(y).
\]
Due to $y^2-y \ln (y)+\ln (y)>y(y-\ln(y))$, the denominator of \eqref{f'A91} is positive, so it is enough to show that $f_{\ref{section35}2}>0$ on $(8,+\infty)$. To this end, we verify that $f_{\ref{section35}2}(8)>0$ and it is strictly increasing on $(8,+\infty)$. To show that it is increasing, we recursively check that its derivative at $y=8$ is positive and increasing on $(8,+\infty)$. After 8 recursive steps of this kind, we arrive at the expression $8 \left(2 \ln ^2(y)+25 \ln (y)+66\right)$, which is clearly positive on $(8,+\infty)$. (During the process, we also put the intermediate results over a common denominator and consider only the numerator for the next step, since the denominator $y$ is positive.)
\end{proof}

\begin{lem}\label{thm23auxlowerest}
For any $x\ge x^*\approx 6288.69$  (defined in Section \ref{refinedapproximationsW}), we have
\[
L_1(x)-\ln
   \left(L_1(x)-L_2(x)+\frac{L_2(x)}{L_1(x)}+\frac{\left(L_2(x)-2\right) L_2(x)}{2 L_1(x){}^2}+ \frac{L_2(x){}^3}{L_1(x){}^3} \right)>
\]
\[
L_1(x)-L_2(x)+\frac{L_2(x)}{L_1(x)}+\frac{(L_2(x)-2) L_2(x)}{2 L_1^2(x)}-\frac{3L_2^2(x)}{2L_1^3(x)}.
\]
\end{lem}
\begin{proof} The proof is analogous to (but more technical than) that of Lemma \ref{thm23auxupperest}, hence it is omitted for brevity.
\end{proof}

\end{document}